\documentclass{article}
\usepackage{amsmath,amsthm,amssymb,cases,amscd}
\usepackage{indentfirst}
\usepackage{ascmac}
\usepackage{yhmath}
\usepackage{mathrsfs}
\usepackage{textcomp}
\usepackage[toc,page]{appendix}
\usepackage{braket}
\usepackage{enumerate}
\usepackage{multienum}
\usepackage{mathtools}

\makeatletter
    
    \@addtoreset{equation}{section}
\makeatother

\usepackage[top=30truemm,bottom=30truemm,left=20truemm,right=20truemm]{geometry} 
\allowdisplaybreaks

\theoremstyle{plain}
\newtheorem{defn}{Definition}[section]
\newtheorem{lem}{Lemma}[section]
\newtheorem{prop}[lem]{Proposition}
\newtheorem{thm}[lem]{Theorem}
\newtheorem{cor}[lem]{Corollary}

\theoremstyle{definition}
\newtheorem*{rem}{Remark}
\newtheorem*{notation}{Notation}
\newtheorem*{acknowledgment}{Acknowledgment}

\newcommand{\lra}{\longrightarrow}

\newcommand{\into}{\hookrightarrow}

\newcommand{\image}{\operatorname{im}}
\newcommand{\myim}{\operatorname{Im}}
\newcommand{\myre}{\operatorname{Re}}

\newcommand{\A}{{\mathbb{A}}}
\newcommand{\C}{\mathbb{C}}
\newcommand{\Q}{\mathbb{Q}}
\newcommand{\R}{\mathbb{R}}
\newcommand{\T}{\mathbb{T}}

\newcommand{\Z}{\mathbb{Z}}
\newcommand{\bfa}{\mathbf{a}}
\newcommand{\bfe}{\mathbf{e}}
\newcommand{\bfk}{\mathbf{k}}

\newcommand{\calD}{\mathcal{D}}
\newcommand{\calF}{\mathcal{F}}
\newcommand{\calH}{\mathcal{H}}
\newcommand{\calL}{\mathcal{L}}
\newcommand{\calO}{\mathcal{O}}
\newcommand{\calS}{\mathcal{S}}
\newcommand{\calT}{\mathcal{T}}
\newcommand{\calZ}{\mathcal{Z}}

\newcommand{\itH}{\mathit{H}}
\newcommand{\itN}{\mathit{N}}

\newcommand{\frakH}{\mathfrak{H}}
\newcommand{\frakS}{\mathfrak{S}}
\newcommand{\frakU}{\mathfrak{U}}
\newcommand{\frakh}{\mathfrak{h}}
\newcommand{\frakk}{\mathfrak{k}}
\newcommand{\frakp}{\mathfrak{p}}
\newcommand{\frakq}{\mathfrak{q}}
\newcommand{\frakt}{\mathfrak{t}}
\newcommand{\fraku}{\mathfrak{u}}

\newcommand{\an}[1]{\langle #1 \rangle}

\newcommand{\sgn}{\operatorname{sgn}}
\newcommand{\fin}{\mathit{fin}}

\newcommand{\M}{\operatorname{M}}
\newcommand{\GL}{\operatorname{GL}}
\newcommand{\PGL}{\operatorname{PGL}}
\newcommand{\Sp}{\operatorname{Sp}}
\newcommand{\GSp}{\operatorname{GSp}}
\newcommand{\PGSp}{\operatorname{PGSp}}
\newcommand{\Mp}{\operatorname{Mp}}
\newcommand{\SO}{\operatorname{SO}}

\newcommand{\SL}{\operatorname{SL}}

\newcommand{\U}{\operatorname{U}}
\newcommand{\Lie}{\operatorname{Lie}}

\newcommand{\fraksp}{\mathfrak{sp}}

\newcommand{\mmatrix}[4]{\begin{pmatrix} #1 & #2 \\ #3 & #4 \end{pmatrix}}
\newcommand{\tp}[1]{\prescript{\mathrm t}{}{#1}}
\newcommand{\diag}{\operatorname{diag}}
\newcommand{\tr}{\operatorname{tr}}

\newcommand{\cent}{\operatorname{Cent}}

\newcommand{\SW}{\mathit{SW}}

\newcommand{\spin}{\mathrm{spin}}

\newcommand{\Hom}{\operatorname{Hom}}
\newcommand{\End}{\operatorname{End}}

\newcommand{\Sym}{\operatorname{Sym}}

\newcommand{\Ind}{\operatorname{Ind}}

\newcommand{\inv}{^{-1}}
\newcommand{\wtil}{\widetilde}
\newcommand{\what}{\widehat}

\title{Proofs of Ibukiyama's conjectures on Siegel modular forms of half-integral weight and of degree 2}
\date{[\today]}
\author{Hiroshi Ishimoto}

\begin{document}

\maketitle

\begin{abstract}
We prove Ibukiyama's conjectures on Siegel modular forms of half-integral weight and of degree 2 by using Arthur's multiplicity formula on the split odd special orthogonal group $\SO_5$ and Gan-Ichino's multiplicity formula on the metaplectic group $\Mp_4$.
In the proof, the representation theory of the Jacobi groups also plays an important role.
\end{abstract}

\setcounter{tocdepth}{1}
\tableofcontents

\section{Introduction} \label{Intr}
In a paper \cite{shi}, Shimura revolutionized the study of half-integral weight modular forms by establishing a lifting from Hecke eigenforms of half-integral weight to Hecke eigenforms of integral weight.
Kohnen \cite{koh} introduced the plus space to show that Shimura's lifting gives an isomorphism
\begin{align*}
S_{k+\frac{1}{2}}^+(\Gamma_0(4))
\cong S_{2k}(\SL_2(\Z)).
\end{align*}
Here the left hand side is Kohnen's plus space, which means a subspace of ``level $\SL_2(\Z)$''.
Later, Ibukiyama \cite{ibuconj, iburef} proposed conjectures on vector valued Siegel modular forms of degree 2 of half-integral weight.
Some of them are similar to Shimura's isomorphism.

Our aim in this paper is to prove Ibukiyama's conjectures by using Gan-Ichino's multiplicity formula.
The conjectures are stated as follows.
The first one is a Shimura type conjecture on vector valued Siegel modular forms of degree 2 of half-integral weight with character (Neben type).
\begin{thm}[\cite{ibuconj}]\label{isc}
For any natural number $k\geq3$ and any even integer $j\geq0$, there is a linear isomorphism
\begin{align*}
S_{\det^{k-\frac{1}{2}} \Sym_j}^+ \left(\Gamma_0(4), \left(\frac{-1}{\cdot}\right) \right)
\cong S_{\det^{j+3} \Sym_{2k-6}} \left( \Sp_4(\Z) \right),
\end{align*}
which preserves $L$-functions.
\end{thm}
Here the superscript $+$ is a generalization of Kohnen's plus space.
This means a ``level $\Sp_4(\Z)$'' part called the plus space.

The next one is a lifting to vector valued Siegel modular forms of degree 2 of half-integral weight without character (Haupt type).
\begin{thm}[{\cite[Conjecture 1.1]{iburef}}]\label{ilc}
For any integer $k\geq0$ and any even integer $j\geq0$, there exists an injective linear map
\begin{equation*}
\calL : S_{2k-4}(\SL_2(\Z)) \otimes S_{2k+2j-2}(\SL_2(\Z)) \lra S_{\det^{k-\frac{1}{2}} \Sym_j}^+(\Gamma_0(4))
\end{equation*}
such that if $f \in S_{2k-4}(\SL_2(\Z))$ and $g \in S_{2k+2j-2}(\SL_2(\Z))$ are Hecke eigenforms, then so is $\calL(f\otimes g) \in S_{\det^{k-\frac{1}{2}} \Sym_j}^+(\Gamma_0(4))$, and they satisfy
\begin{equation*}
L(s, \calL(f\otimes g)) = L(s-j-1, f) L(s,g).
\end{equation*}
\end{thm}

We shall write $S_{\det^{k-\frac{1}{2}} \Sym_j}^{+,0}(\Gamma_0(4))$ for the orthogonal complement of the image of the injective map $\calL$.
\begin{thm}[{\cite[Conjecture 1.2]{iburef}}]\label{icc}
For any integer $k\geq3$ and any even integer $j\geq0$, there exists a linear isomorphism
\begin{equation*}
\rho^\itH : S_{\det^{k-\frac{1}{2}} \Sym_j}^{+,0}(\Gamma_0(4)) \lra S_{\det^{j+3}\Sym_{2k-6}}(\Sp_4(\Z)),
\end{equation*}
which preserves $L$-functions.
\end{thm}

The details of the notation and the definitions of $L$-functions will be reviewed in \S \ref{Main} and \S \ref{Jacobi form}, and we shall state the conjecture again.
In order to prove Ibukiyama's conjectures, we will apply the multiplicity formulae of Arthur and Gan-Ichino, and the representation theory of the Jacobi groups.

Let us recall the history and the statement of the multiplicity formulae briefly.
After Shimura's result \cite{shi}, Waldspurger \cite{wal1, wal2} studied Shimura's correspondence in the framework of automorphic representations of the metaplectic group $\Mp_2$, which is a nonlinear two-fold covering group of $\Sp_2=\SL_2$.
Namely, he described the automorphic discrete spectrum of $\Mp_2$ in terms of that of $\SO_3=\PGL_2$ via the global theta lifts between $\Mp_2$ and (inner forms of) $\SO_3$.
About 30 years after that, Gan-Ichino \cite{gi18, gi20} studied the automorphic discrete spectrum $L^2_{\mathrm{disc}}(\Mp_{2n})$ of the metaplectic group $\Mp_{2n}$, which is a nonlinear two-fold covering group of the symplectic group $\Sp_{2n}$ of rank $n$, and  partially proved the following decomposition predicted by Gan \cite{gan}:
\begin{align}\label{gi}
L^2_{\mathrm{disc}}(\Mp_{2n})
=\bigoplus_\phi \bigoplus_\eta m_{\phi,\eta} \pi_{\phi,\eta},
\end{align}
where $\phi$ and $\eta$ run over elliptic $A$-parameters for $\SO_{2n+1}$ and continuous characters of the adelic component group $S_{\phi, \A}$ of $\phi$ respectively, and $m_{\phi,\eta}$ and $\pi_{\phi,\eta}$ are the corresponding nonnegative integer and irreducible automorphic representation, respectively.
Strictly speaking, $m_{\phi,\eta}$ and $\pi_{\phi,\eta}$ depend on an auxiliary choice of an additive character of the ring of adeles.
Their works can be considered as a generalization of the work of Waldspurger.
In \cite{gi18}, they partially proved the decomposition \eqref{gi} for any rank $n$, and later, in the other paper \cite{gi20}, they completely proved it in the case that the rank $n$ is 2.
Note that thus the decomposition \eqref{gi} is still conjecture in general.
In both papers, they used the theta correspondence and Arthur's result \cite{art}.
Here, the result of Arthur that Gan-Ichino used is a decomposition
\begin{align*}
L^2_{\mathrm{disc}}(\SO_{2n+1})
=\bigoplus_\phi \bigoplus_\eta n_{\phi,\eta} \sigma_{\phi,\eta}
\end{align*}
of the automorphic discrete spectrum $L^2_{\mathrm{disc}}(\SO_{2n+1})$ of the split odd special orthogonal group $\SO_{2n+1}$, where $\phi$ and $\eta$ run over elliptic $A$-parameters for $\SO_{2n+1}$ and continuous characters of the adelic component group $S_{\phi, \A}$ of $\phi$ respectively, and $n_{\phi,\eta}$ and $\sigma_{\phi,\eta}$ are the corresponding nonnegative integer and irreducible automorphic representation, respectively.
We review the multiplicity formulae for the automorphic discrete spectrum of $\SO_5$ given by Arthur and of $\Mp_4$ given by Gan-Ichino in \S \ref{MF} more precisely.

In \S \ref{Jacobi form}, we complete the definition of the $L$-function of a Siegel modular form of half-integral weight in the plus space $S_{\det^{k-\frac{1}{2}} \Sym_j}^+(\Gamma_0(4), \left(\frac{-1}{\cdot}\right)^l)$.
Since the elements in the plus space are of level $\Gamma_0(4)$, the Euler factor of $L$-function at 2 cannot be defined in the usual way involving Hecke operators.
However, we can define the Euler 2-factor of $L$-function by using Hecke operators on the Jacobi forms of level $\Sp^J_4(\Z)$, by virtue of the following canonical isomorphism.
Here $\Sp^J_4$ denote the Jacobi group of degree 2.
In the beginning of the section, we define holomorphic and skew-holomorphic Jacobi forms and recall from \cite{ibuconj} a canonical isomorphism (Theorem \ref{5.1})
\begin{align}\label{5.1.0}
J_{\det^k \Sym_j,1}^{\star, \mathrm{cusp}} \overset{\simeq}{\lra} S_{\det^{k-\frac{1}{2}} \Sym_j}^+(\Gamma_0(4), \left(\frac{-1}{\cdot}\right)^l),
\end{align}
between a space of holomorphic or skew-holomorphic Jacobi cusp forms of level $\Sp^J_4(\Z)$ and the plus space.
Here, $\star$ means ``holomorphic'' if $k+l$ is even, and ``skew-holomorphic'' if it is odd.
In the rest of the section, we state the definition of the Hecke operators acting on Jacobi forms.
Then we can define the Hecke operators at 2 on Siegel modular forms in the plus space by the pullback of the Hecke operators on Jacobi forms.

The next section \S \ref{Jacobi group} plays the key role in this paper.
In order to utilize the multiplicity formulae of Arthur and Gan-Ichino for proving the main theorems (Theorems \ref{isc}, \ref{ilc}, and \ref{icc}), we should consider the adelic lifts of Hecke eigenforms in $S_{\det^{j+3} \Sym_{2k-6}} \left( \Sp_4(\Z) \right)$ and $S_{\det^{k-\frac{1}{2}} \Sym_j}^+ \left(\Gamma_0(4), \left(\frac{-1}{\cdot}\right)^l \right)$ to $\SO_5(\A_\Q)$ and $\Mp_4(\A_\Q)$, respectively.
Nevertheless, it is difficult to study the adelic lifts of Hecke eigenforms in the plus space to $\Mp_4(\A_\Q)$.
They are expected to be unramified everywhere, but we cannot show that they are unramified at $p=2$ in the usual way.
This is because the level is $\Gamma_0(4)$ and there is no notion of spherical representations of $\Mp_{2n}(\Q_2)$.
(The adelic interpretation of the plus spaces is already known by \cite{hi, su}, but the author do not know how to apply it to our aim.)
However, we have the isomorphism \eqref{5.1.0} and we can also see that the adelic lifts of eigenforms in $J_{\det^k \Sym_j,1}^{\star, \mathrm{cusp}}$ is unramified everywhere by the usual argument, since the level is $\Sp^J_4(\Z)$.
Thus we need the representation theory of the Jacobi group $\Sp^J_4$.
In \S \ref{Jacobi group}, we introduce the representation theory of the Jacobi groups $\Sp^J_{2n}$ of general degree over local fields and rings of adeles.
Roughly speaking, we shall prove that there is a natural bijection between representations of $\Sp^J_{2n}$ with a fixed nontrivial central character and genuine representations of $\Mp_{2n}$, and this bijective correspondence preserves some canonical properties, such as unramifiedness.
Consequently, we can obtain automorphic representations of $\Mp_4(\A_\Q)$ from Jacobi forms and utilize Gan-Ichino's multiplicity formula.
As for the representation theory of the Jacobi groups, the case of degree 2 is enough for our aim, but the arguments hold in general case and some of them are originated in this paper, so we shall do for general degree.

In \S \ref{Adelic lift}, we consider the adelic lifts of Siegel modular forms and those of Jacobi forms, and prove Theorems \ref{isc}, \ref{ilc}, and \ref{icc}.
On the one hand, we consider the adelic lift of each Hecke eigenform in $S_{\det^{j+3} \Sym_{2k-6}} \left( \Sp_4(\Z) \right)$ to $\GSp_4(\A_\Q)$ in the usual way.
It can be seen that the adelic lift is an irreducible cuspidal automorphic representation of $\GSp_4(\A_\Q)$ with trivial central character, and it is known that there is an accidental isomorphism $\PGSp_4 \cong \SO_5$, so we obtain an irreducible cuspidal automorphic representation of $\SO_5(\A_\Q)$.
Thus by Arthur's multiplicity formula, we can attach an elliptic $A$-parameter for $\SO_5$ to any Hecke eigenform in $S_{\det^{j+3} \Sym_{2k-6}} \left( \Sp_4(\Z) \right)$.
Moreover, we shall show in Lemmas \ref{api} and \ref{iap} that the attached $A$-parameters are tempered and of the form
\begin{align}\label{12}
\phi=\tau \boxtimes S_1,
\end{align}
where $\tau$ is an irreducible cuspidal symplectic automorphic representation of $\GL_4(\A_\Q)$ with certain conditions and $S_1$ denotes the trivial representation of $\SL_2(\C)$, and this attachment is bijective in some sense.
On the other hand, we show that the adelic lift of a Hecke eigenform in $J_{\det^k \Sym_j,1}^{\star, \mathrm{cusp}}$ is an irreducible cuspidal automorphic representation of the Jacobi group $\Sp^J_4(\A_\Q)$.
By the result of \S \ref{Jacobi group}, we obtain an irreducible cuspidal automorphic representation of the metaplectic group $\Mp_4(\A_\Q)$.
In addition, by Gan-Ichino's multiplicity formula, we can attach an elliptic $A$-parameter for $\SO_5$ to a Hecke eigenform in $J_{\det^k \Sym_j,1}^{\star, \mathrm{cusp}}$.
Then in Lemmas \ref{apj} and \ref{jap} we shall show that the attached $A$-parameters are tempered and of the form \eqref{12} or
\begin{align}\label{13}
\phi=\sigma \boxtimes S_1 \oplus \sigma' \boxtimes S_1,
\end{align}
where $\sigma$ and $\sigma'$ are irreducible cuspidal symplectic automorphic representations of $\GL_2(\A_\Q)$, and this attachment is bijective in some sense.
More precisely, combining this with the isomorphism \eqref{5.1.0}, we shall see that the plus space $S_{\det^{k-\frac{1}{2}} \Sym_j}^+ \left(\Gamma_0(4), \left(\frac{-1}{\cdot}\right) \right)$ of Neben type corresponds to the $A$-parameters of the form \eqref{12}, and that the plus space $S_{\det^{k-\frac{1}{2}} \Sym_j}^+ \left(\Gamma_0(4)\right)$ of Haupt type corresponds to the both $A$-parameters of the forms \eqref{12} and \eqref{13}.
Since Hecke eigenforms $f \in S_{2k-4}(\SL_2(\Z))$ and $g \in S_{2k+2j-2}(\SL_2(\Z))$ give irreducible cuspidal automorphic representations $\sigma_f$ and $\sigma_g$ of $\GL_2(\A_\Q)$, the image of $\calL$ is the subspace corresponding to the $A$-parameters of the form \eqref{13}.
Then the subspace corresponding to the $A$-parameters of the form \eqref{12} is $S_{\det^{k-\frac{1}{2}} \Sym_j}^{+,0} \left(\Gamma_0(4)\right)$.
Consequently, the main theorems are proved.
As mentioned above, the reason why we have to consider the Jacobi forms and the representations of the Jacobi groups is that it is difficult to actually see that the plus space is a level $\Sp_4(\Z)$ part by studying only the plus space.
In other words, it is difficult to show that the adelic lift of a Hecke eigenform in the plus space $S_{\det^{k-\frac{1}{2}} \Sym_j}^+ \left(\Gamma_0(4), \left(\frac{-1}{\cdot}\right)^l \right)$ to $\Mp_4(\A_\Q)$ is unramified at a finite place 2.
However, it is easier to show that the adelic lift of a Jacobi form of level $\Sp^J_4(\Z)$ is unramified everywhere.
Thus the theory of the relation between representations of the Jacobi group and those of the metaplectic group is the key ingredient of this paper.\\

Although the arguments in appendix \ref{Appendix} are not needed to prove our main theorem, the result is so interesting that it is contained in this paper.
In the appendix, we describe the relation between the adelic lifts of Siegel modular forms of half-integral weight and those of Jacobi forms corresponding to them by the canonical isomorphism \eqref{5.1.0}.
Thanks to the multiplicity formula of Gan-Ichino and our results on the representations of the Jacobi groups given in \S \ref{Jacobi group}, we can show not only that the adelic lifts of Siegel modular forms in the plus space is unramified at a finite place 2, but also Theorem \ref{appthm}, which asserts that the canonical isomorphism \eqref{5.1.0} is compatible with a natural bijection between representations of $\Mp_4$ and those of $\Sp^J_4$ given in \S \ref{Jacobi group}.
Theorem \ref{appthm} seems not so natural but canonical and interesting.

\begin{notation}
If $F$ is a number field, we write $\A_F$ for the ring of the adeles of $F$.
In particular, if $F$ is the field $\Q$ of rational numbers, let $\A_\Q$ be abbreviated as $\A$.
We write $\calS(\A_F^n)$ for the Schwartz space on $\A_F^n$.
For a nontrivial additive character $\psi$ of $F\backslash \A_F$, put $\overline{\psi}(x)=\overline{\psi(x)}$ and we call the additive character $\overline{\psi}$ its complex conjugate.
Moreover, if $v$ is a place of $F$, we write $\psi_v$ for the local component of $\psi$ at $v$.

If $F$ is a local field, let $W_F$ be the Weil group of $F$ and put
\begin{align*}
L_F=
\begin{cases}
W_F \times \SL_2(\C),&\text{if $F$ is nonarchimedean},\\
W_F,&\text{if $F$ is archimedean}.
\end{cases}
\end{align*}
We write $\calS(F^n)$ for the Schwartz space on $F^n$.

We shall write $e_i$ for the column vector with entry 1 at $i$-th position and zeros elsewhere, and $E_{i,j}$ for the matrix with entry 1 at $(i,j)$-th position and zeros elsewhere.
For any representation $\pi$, let $\pi^\vee$ denote its contragredient representation.
For any finite abelian group $S$, let $\what{S}$ denote the group of characters of $S$.
For any positive integer $d\geq1$, we write $S_d$ for the unique irreducible representation of dimension $d$ of $\SL_2(\C)$.
For $z \in \C$, we set $\bfe(z)=\exp(2\pi i z)$.
In this paper, trivial characters are included in quadratic characters.
\end{notation}

\begin{acknowledgment}
The author thanks his supervisor Atsushi Ichino for suggesting the problem and for his helpful advice, and Tamotsu Ikeda and Hiraku Atobe for sincere and useful comments.
The author also thanks Hiraku Atobe and Tomoyoshi Ibukiyama for bringing a book \cite{cl} and a paper \cite{iburef}, respectively, to the attention of the author.
This work was partially supported by JSPS Research Fellowships for Young Scientists KAKENHI Grant Number 20J11779.
\end{acknowledgment}


\section{Main theorem}\label{Main}
In this section, after reviewing the definitions of several types of cusp forms and their $L$-functions, we state Ibukiyama's conjectures.

\subsection{Review on elliptic cusp forms}
In this subsection, we shall review the notation and the definition of elliptic cusp forms.

For an integer $n\geq1$, let $\frakH_n$ denote the Siegel upper half space of degree $n$, i.e.,
\begin{equation*}
\frakH_n
=\Set{Z=X+iY \in \M_n(\C) | X=\tp{X}, Y=\tp{Y} \in \M_n(\R), \ Y >0 }.
\end{equation*}
Here, $Y>0$ means that $Y$ is positive definite.
We write $X=\myre(Z)$, and $Y=\myim(Z)$, for $Z=X+iY\in\frakH_n$, and call the imaginary part and the real part, respectively.
By abuse of notation, we shall write $i$ for the scalar matrix in $\frakH_n$ of the diagonal entries $i$.
For any commutative ring $R$ with unity 1, let $\Sp_{2n}(R)$ be the symplectic group over $R$ of degree $n$, and $\GSp_{2n}(R)$ the symplectic similitude group over $R$ of degree $n$ i.e.,
\begin{align*}
\Sp_{2n}(R)
&=\Set{g \in \GL_{2n}(R) | \tp{g} J_n g = J_n},\\
\GSp_{2n}(R)
&=\Set{ g\in\GL_{2n}(R)| \tp{g}J_ng=\nu(g)J_n, \ \exists\nu(g)\in \GL_1(R)},
\end{align*}
where $J_n=\mmatrix{0}{1_n}{-1_n}{0}$.
Note that $\Sp_2=\SL_2$ and $\GSp_2=\GL_2$.

The group $\GSp_{2n}^+(\R)=\{g \in \GSp_{2n}(\R)|\nu(g)>0\}$ acts on $\frakH_n$ in the following way:
\begin{align*}
gZ &=(AZ+B)(CZ+D)\inv,& g&=\mmatrix{A}{B}{C}{D} \in \GSp_{2n}^+(\R), \quad Z\in\frakH_n.
\end{align*}
We set a factor of automorphy $J(g,Z)=CZ+D$.\\

Elliptic cusp forms are defined as follows.
\begin{defn}
Let $f(z)$ be a $\C$-valued function on $z\in \frakH_1$.
When $f$ satisfies the next conditions (0)-(2), we say that $f$ is an elliptic cusp form of weight $k$.
\begin{enumerate}[(1)]
\setcounter{enumi}{-1}
\item $f$ is holomorphic;
\item $J(\gamma, z)^{-k} f(\gamma z)=f(z)$ for all $\gamma \in \SL_2(\Z)$;
\item $f$ has a Fourier expansion of the following form:
\begin{align*}
f(z)
=\sum_{n\in\Z_{>0}} a_n \bfe(nz),
\end{align*}
where $\Z_{>0}$ denotes the set of all positive integers.
\end{enumerate}
We write $S_k(\SL_2(\Z))$ for the space of such functions.
\end{defn}

The Hecke operator $T(p)$ on $S_k(\SL_2(\Z))$ at a prime number $p$ is defined by
\begin{align*}
[f|_k T(p)] (z)
=p^{k-1} \sum_g J(g,z)^{-k} f(gz),
\end{align*}
where $g$ runs over a complete representative system of $\SL_2(\Z) \backslash \SL_2(\Z) \mmatrix{1}{0}{0}{p}\SL_2(\Z)$.

For a Hecke eigenform $f \in S_k(\SL_2(\Z))$, let $c_p$ denotes the eigenvalue of $T(p)$ associated with $f$.
Then the $L$-function of $f$ is defined as an Euler product
\begin{equation*}
L(s,f)
=\prod_p ( 1 - c_p p^{-s} + p^{k-1-2s} )\inv.
\end{equation*}

Finally, let us recall the Petersson inner product on $S_k(\SL_2(\Z))$.
For $f_1$, $f_2 \in S_k(\SL_2(\Z))$, it is defined by
\begin{equation*}
\an{f_1, f_2}
=\int_{\SL_2(\Z) \backslash \frakH_1} f_1(z) \overline{f_2(z)} y^{k-2} dxdy,
\end{equation*}
where $x$ and $y$ are the real and imaginary part of $z$, respectively.

It is known that the Hecke operators are Hermite with respect to the Petersson inner product, and the $\C$-vector space $S_k(\SL_2(\Z))$ has an orthonormal basis consisting of Hecke eigenforms.
\subsection{Vector valued Siegel cusp forms of integral weight}
In order to state the definition of vector valued Siegel modular forms of degree 2, we recall rational irreducible representations of $\GL_2(\C)$ here.
For any integer $j\geq0$, we write $(\Sym_j, V_j)$ for the symmetric tensor representation of degree $j$ (i.e., dimension $j+1$) of $\GL_2(\C)$.
Any rational irreducible representation of $\GL_2(\C)$ can be written as $\det^k \otimes \Sym_j$ by using an integer $k$ and a nonnegative integer $j$.
For any $V_j$-valued function $F(Z)$ on $\frakH_2$ and $g\in\GSp_4^+(\R)$, we define the slash operator of weight $(k,j)$ by
\begin{align*}
\left[F|_{k,j}g\right] (Z)
=\nu(g)^{k+\frac{j}{2}} \det(J(g,Z))^{-k} \Sym_j(J(g,Z))\inv F(gZ).
\end{align*}

\begin{defn}
Let $F(Z)$ be a $V_j$-valued function on $Z\in \frakH_2$.
When $F$ satisfies the next conditions (0)-(2), we say that $F$ is a Siegel cusp form of weight $(k,j)$.
\begin{enumerate}[(1)]
\setcounter{enumi}{-1}
\item $F$ is holomorphic;
\item $F|_{k,j} \gamma =F$ for all $\gamma \in \Sp_4(\Z)$;
\item $F$ has a Fourier expansion of the following form:
\begin{align*}
F(Z)
=\sum_{\substack{T \in L_2^* \\ T>0}} A(T) \bfe(\tr(TZ)),
\end{align*}
where $L_2^*$ denotes the set of all half-integral symmetric matrices of size $2\times2$.
\end{enumerate}
We write $S_{k,j}(\Sp_4(\Z))$ for the space of such functions.
\end{defn}

Now we define Hecke operators $T(m)$ and the spinor $L$-functions.
For a natural number $m$, put
\begin{align*}
F|_{k,j} T(m)
=m^{k+\frac{j}{2}-3} \sum_{g \in \Sp_4(\Z) \backslash X(m)} F|_{k,j}g,
\end{align*}
where
\begin{align*}
X(m)
=\set{ x \in \M_4(\Z) \cap \GSp_4^+(\R) | \nu(x)=m },
\end{align*}
and for a Hecke eigenform $F \in S_{k,j}(\Sp_4(\Z))$, let $\lambda(m)$ denotes the eigenvalue of $T(m)$ associated with $F$.
Note that $T(mn)=T(m)T(n)$ if $m, n \in \Z_{>0}$ are coprime.
The spinor $L$-function $L(s,F,\spin)$ of a Hecke eigenform $F\in S_{k,j}(\Sp_4(\Z))$ is defined as an Euler product
\begin{equation*}
L(s,F,\spin)
=\prod_p L(s,F,\spin)_p,
\end{equation*}
where the product runs over all prime numbers $p$,
\begin{equation*}
L(s,F,\spin)_p
=\left(
1 -\lambda(p)p^{-s} +(\lambda(p)^2-\lambda(p^2)-p^{\mu-1})p^{-2s} -\lambda(p)p^{\mu-3s} +p^{2\mu-4s}
\right)\inv,
\end{equation*}
and $\mu=2k+j-3$.

As in the case of elliptic cusp forms, it is known that the Petersson inner product can be defined on $S_{k,j}(\Sp_4(\Z))$, and the $\C$-vector space $S_{k,j}(\Sp_4(\Z))$ has an orthonormal basis consisting of Hecke eigenforms.
\subsection{Vector valued Siegel cusp forms of half-integral weight}
Now we come to vector valued Siegel cusp forms of half-integral weight.
Put
\begin{equation*}
\Gamma_0(4)
=\Set{\gamma=\mmatrix{A}{B}{C}{D} \in \Sp_4(\Z) | C\equiv0 \mod4},
\end{equation*}
and
\begin{equation*}
\theta(Z)
=\sum_{x\in\Z^2} \bfe(\tp{x} Z x),
\quad Z\in\frakH_2.
\end{equation*}
Let us define a character $\left(\frac{-1}{\cdot}\right)$ of $\Gamma_0(4)$ by
\begin{align*}
\left(\frac{-1}{\gamma}\right)
=\left(\frac{-1}{\det D}\right),
\quad
\gamma
=\left(\begin{array}{cc}
A&B\\
C&D
\end{array}\right) \in \Gamma_0(4).
\end{align*}

In order to state the definition of vector valued Siegel modular forms of half-integral weight, we recall the slash operator of half-integral weight.
First we define a covering group of $\GSp_4^+(\R)$.
Let $\wtil{\GSp_4^+}(\R)$ denote the set of pairs $(g, \phi(Z))$, where $g$ is in $\GSp_4^+(\R)$ and $\phi(Z)$ is a $\C$-valued holomorphic function in $Z\in\frakH_2$ such that $\phi(Z)^4=\frac{\det J(g,Z)^2}{\det g}$.
The set $\wtil{\GSp_4^+}(\R)$ is a group under the multiplication defined by
\begin{equation*}
(g_1,\phi_1(Z)) (g_2, \phi_2(Z))
=(g_1 g_2, \phi_1(g_2Z)\phi_2(Z)).
\end{equation*}
The group $\wtil{\GSp_4^+}(\R)$ can be regarded as a 4-fold covering group of $\GSp_4^+(\R)$ by $(g, \phi(Z))\mapsto g$. 
We can embed $\Gamma_0(4)$ in $\wtil{\GSp_4^+}(\R)$ by
\begin{align*}
\gamma \mapsto \left(\gamma, \frac{\theta(\gamma Z)}{\theta(Z)} \right)
\end{align*}
and we write $\wtil{\Gamma}_0(4)$ for the image of $\Gamma_0(4)$.
In this way, we shall identify $\Gamma_0(4)$ with the subgroup $\wtil{\Gamma}_0(4)$ of $\wtil{\GSp_4^+}(\R)$.
Let $j\geq0$ be an integer.
Then for any $V_j$-valued function $F(Z)$ on $\frakH_2$ and $(g,\phi(Z))\in\wtil{\GSp_4^+}(\R)$, we define the slash operator of weight $(k-\frac{1}{2},j)$ by
\begin{align*}
\left[F|_{k-\frac{1}{2},j} (g, \phi(Z)) \right] (Z)
=\nu(g)^\frac{j}{2} \phi(Z)^{-2k+1} \Sym_j(J(g,Z))\inv F(gZ).
\end{align*}

\begin{defn}
Let $F(Z)$ be a $V_j$-valued function on $Z\in \frakH_2$.
When $F$ satisfies the next conditions (0)-(2), we say that $F$ is a Siegel cusp form of weight $(k-\frac{1}{2},j)$, level $\Gamma_0(4)$, character $\left(\frac{-1}{\cdot}\right)^l$ ($l\in(\Z/2\Z)$).
\begin{enumerate}[(1)]
\setcounter{enumi}{-1}
\item $F$ is holomorphic;
\item $F|_{k-\frac{1}{2},j} \gamma = \left(\frac{-1}{\gamma}\right)^lF$ for all $\gamma \in \Gamma_0(4)$;
\item $F$ has a Fourier expansion of the following form:
\begin{align*}
F(Z)
=\sum_{\substack{T \in L_2^* \\ T>0}} A(T) \bfe(\tr(TZ)),
\end{align*}
\end{enumerate}
The space of such functions will be denoted by $S_{k-\frac{1}{2},j}(\Gamma_0(4), \left(\frac{-1}{\cdot}\right)^l)$.
Moreover, we write $S_{k-\frac{1}{2},j}^+(\Gamma_0(4), \left(\frac{-1}{\cdot}\right)^l)$ for the subspace consisting of those $F\in S_{k-\frac{1}{2},j}(\Gamma_0(4), \left(\frac{-1}{\cdot}\right)^l)$ such that $A(T)=0$ unless $T\equiv(-1)^{k+l-1} r \tp{r} \mod4L_2^*$ for some $r \in \Z^2$.
The subspace $S_{k-\frac{1}{2},j}^+(\Gamma_0(4), \left(\frac{-1}{\cdot}\right)^l)$ is called the plus space.
\end{defn}

When $l$ is even, we often write $S_{k-\frac{1}{2},j}^+(\Gamma_0(4))$ for $S_{k-\frac{1}{2},j}^+(\Gamma_0(4), \left(\frac{-1}{\cdot}\right)^l)$.

Now we define Hecke operators $T_s(p)$ and the $L$-functions of Hecke eigenforms of half-integral weight.
For any prime number $p$, we put
\begin{align*}
&
K_s(p^2)
=\left(\begin{array}{cccc}
1_{2-s}&&&\\
&p1_s&&\\
&&p^21_{2-s}&\\
&&&p1_s
\end{array}\right),
&
&
s=0,1,2.
&
\end{align*}
Let $F \in S_{k-\frac{1}{2}, j}(\Gamma_0(4), \left(\frac{-1}{\cdot}\right)^l)$.
Then for any odd prime $p$, we define the Hecke operators
\begin{align*}
F|_{k-\frac{1}{2},j} T_s(p)
=\sum_t \left(\frac{-1}{\wtil{g}_{s,t}}\right)^l F|_{k-\frac{1}{2}, j} \wtil{g}_{s,t},
\qquad
s=0, 1,
\end{align*}
by using the right coset decompositions
\begin{align*}
\wtil{\Gamma}_0(4) (K_s(p^2), p^{1-\frac{s}{2}}) \wtil{\Gamma}_0(4)
=\bigsqcup_t \wtil{\Gamma}_0(4) \wtil{g}_{s,t}
\end{align*}
of the double cosets of $\wtil{\Gamma}_0(4)$ containing $K_s(p^2)$.
Note that although $\wtil{g}_{s,t}$ is not an element of $\Gamma_0(4)$ for some $t$, we can define $\left(\frac{-1}{\wtil{g}_{s,t}}\right)$ in the natural way.
When $p=2$, if $F$ is in the plus space $S_{k-\frac{1}{2}, j}^+(\Gamma_0(4), \left(\frac{-1}{\cdot}\right)^l)$, we can also define the Hecke operators $T_s(2)$ on $F$ by using those of holomorphic or skew-holomorphic Jacobi forms.
The precise definition will be given later in \S\ref{Jacobi form}.
For a Hecke eigenform $F \in S_{k-\frac{1}{2}, j}^+(\Gamma_0(4), \left(\frac{-1}{\cdot}\right)^l)$, let $\eta(p)$ and $\omega(p)$ denote the eigenvalues of $T_1(p)$ and $T_0(p)$ associated with $F$, respectively.
The Euler $p$-factor of $L$-function of such $F$ for a prime $p$ is given by
\begin{equation*}
L(s,F)_p
=\left(
1 -\eta^*(p)p^{-s} +(p\omega^*(p) +p^{\nu-2}(1+p^2))p^{-2s} -\eta^*(p)p^{\nu-3s} +p^{2\nu-4s}
\right)\inv,
\end{equation*}
where $\eta^*(p)=\left(\frac{-1}{p}\right)^l p^{k+j-\frac{7}{2}} \eta(p)$, $\omega^*(p)=p^{2k+2j-7}\omega(p)$, and $\nu=2k+2j-3$.
The $L$-function $L(s,F)$ of $F$ is defined as an Euler product
\begin{equation*}
L(s,F)
=\prod_p L(s,F)_p,
\end{equation*}
where the product runs over all prime numbers $p$.

As in the case of elliptic cusp forms, the Petersson inner product can be defined on $S_{k-\frac{1}{2},j}^+(\Gamma_0(4), \left(\frac{-1}{\cdot}\right)^l)$, and it is known that $S_{k-\frac{1}{2},j}^+(\Gamma_0(4), \left(\frac{-1}{\cdot}\right)^l)$ has an orthonormal basis consisting of Hecke eigenforms.
\subsection{Ibukiyama's conjectures}

Now we can state Ibukiyama's conjectures, our main theorems.
\begin{thm}[Shimura type isomorphism on the Neben type]\label{main}
For any integer $k\geq3$ and any even integer $j\geq0$, there exists a linear isomorphism
\begin{equation*}
\rho^\itN : S_{k-\frac{1}{2},j}^+(\Gamma_0(4), \left(\frac{-1}{\cdot}\right)) \lra S_{j+3, 2k-6}(\Sp_4(\Z))
\end{equation*}
such that if $F \in S_{k-\frac{1}{2},j}^+(\Gamma_0(4), \left(\frac{-1}{\cdot}\right))$ is a Hecke eigenform, then so is $\rho^\itN(F) \in S_{j+3, 2k-6}(\Sp_4(\Z))$, and they satisfy
\begin{equation*}
L(s,F) =L(s,\rho^\itN(F), \spin).
\end{equation*}
\end{thm}

\begin{thm}[Lifting to the Haupt type]\label{lifting}
For any integer $k\geq0$ and any even integer $j\geq0$, there exists an injective linear map
\begin{equation*}
\calL : S_{2k-4}(\SL_2(\Z)) \otimes S_{2k+2j-2}(\SL_2(\Z)) \lra S_{k-\frac{1}{2},j}^+(\Gamma_0(4))
\end{equation*}
such that if $f \in S_{2k-4}(\SL_2(\Z))$ and $g \in S_{2k+2j-2}(\SL_2(\Z))$ are Hecke eigenforms, then so is $\calL(f\otimes g) \in S_{k-\frac{1}{2},j}^+(\Gamma_0(4))$, and they satisfy
\begin{equation*}
L(s, \calL(f\otimes g)) = L(s-j-1, f) L(s,g).
\end{equation*}
\end{thm}

Note that if $k\leq7$, then $2k-4\leq10$, and $S_{2k-4}(\SL_2(\Z))=0$.
Thus the lifting theorem holds trivially with $\calL=0$ when $k\leq7$.

Let $S_{k-\frac{1}{2},j}^{+,0}(\Gamma_0(4))$ denote the orthogonal complement of the image of the injective map $\calL$.
\begin{thm}[Shimura type isomorphism on the Haupt type]\label{compl}
For any integer $k\geq3$ and any even integer $j\geq0$, there exists a linear isomorphism
\begin{equation*}
\rho^\itH : S_{k-\frac{1}{2},j}^{+,0}(\Gamma_0(4)) \lra S_{j+3, 2k-6}(\Sp_4(\Z))
\end{equation*}
such that if $F \in S_{k-\frac{1}{2},j}^{+,0}(\Gamma_0(4))$ is a Hecke eigenform, then so is $\rho^\itH(F) \in S_{j+3, 2k-6}(\Sp_4(\Z))$, and they satisfy
\begin{equation*}
L(s,F) =L(s,\rho^\itH(F), \spin).
\end{equation*}
\end{thm}

Our main aim in this paper is to prove them.
By these theorems, we obtain the following corollaries.

\begin{cor}
For any integer $k\geq3$ and any even integer $j\geq0$, there exists a linear isomorphism
\begin{equation*}
S_{k-\frac{1}{2},j}^{+,0}(\Gamma_0(4)) \cong S_{k-\frac{1}{2},j}^+(\Gamma_0(4), \left(\frac{-1}{\cdot}\right))
\end{equation*}
which preserves the $L$-functions.
\end{cor}

\begin{cor}
For any integer $k\geq3$ and any even integer $j\geq0$, we have
\begin{equation*}
\dim S_{k-\frac{1}{2},j}^+(\Gamma_0(4)) - \dim S_{2k-4}(\SL_2(\Z)) \times \dim S_{2k+2j-2}(\SL_2(\Z))
=\dim S_{k-\frac{1}{2},j}^+(\Gamma_0(4), \left(\frac{-1}{\cdot}\right)).
\end{equation*}
\end{cor}
\section{Multiplicity formulae for split $\SO_5$ and $\Mp_4$}\label{MF}
In this section, we recall the multiplicity formulae for split $\SO_5$ \cite{art} and $\Mp_4$ \cite{gi18, gi20}.
They play important roles in this paper.
\subsection{Orthogonal and metaplectic groups}
First, we fix some notations for orthogonal and metaplectic groups.

We shall write $\SO_{2n+1}$ and $\Sp_{2n}$ for the split odd special orthogonal group and the symplectic group of rank $n$, which are defined by
\begin{align*}
\SO_{2n+1}
&=
\Set{h \in \SL_{2n+1} | \tp{h} \left(\begin{array}{cc}1_{n+1}&\\&-1_n\end{array}\right) h = \left(\begin{array}{ccc}1_{n+1}&\\&-1_n\end{array}\right)},\\
\Sp_{2n}
&=
\Set{g \in \SL_{2n} | \tp{g} \left(\begin{array}{cc}&1_n\\-1_n&\end{array}\right) g = \left(\begin{array}{cc}&1_n\\-1_n&\end{array}\right)},
\end{align*}
respectively.

Let $F$ be a local field of characteristic 0.
The 2-cocycle $c_F(-,-):\Sp_{2n}(F)\times \Sp_{2n}(F) \to \{\pm1\}$ given by Ranga Rao \cite[\S5]{rao} defines a double cover
\begin{equation*}
1 \lra \set{\pm1} \lra \Mp_{2n}(F) \lra \Sp_{2n}(F) \lra 1,
\end{equation*}
which is nonlinear unless $F=\C$.
The group $\Mp_{2n}(F)$ is called the metaplectic group over $F$.
We shall identify $\Mp_{2n}(F) = \Sp_{2n}(F) \times \{\pm1\}$ as sets, and then the multiplication law is given by
\begin{equation*}
(g,\epsilon) \cdot (g',\epsilon')
=(gg', \epsilon\epsilon' c_F(g,g')).
\end{equation*}
Moreover, if $F$ is a nonarchimedean local field of residual characteristic other than 2, with the ring of integers $\calO$, then the covering splits over $\Sp_{2n}(\calO)$.
Let us write $s_F$ for this splitting:
\begin{equation}\label{spl}
\Sp_{2n}(\calO) \into \Mp_{2n}(F) , \quad g \mapsto (g, s_F(g)).
\end{equation}

Let $F$ be a number field.
For any place $v$ of $F$, we shall write $c_v=c_{F_v}$ and $s_v=s_{F_v}$, for simplicity.
For a finite set $\mathfrak{S}$ of places of $F$ that contains all places above $\infty$ and 2, put
\begin{equation*}
\Sp_{2n}(\A_F)_\mathfrak{S}
=\prod_{v\in\mathfrak{S}}\Sp_{2n}(F_v) \times \prod_{v \notin \mathfrak{S}}\Sp_{2n}(\calO_v),
\end{equation*}
where $\calO_v$ denotes the ring of integers of $F_v$.
Then we define the double covering $\Mp_{2n}(\A_F)_{\mathfrak{S}_1}\to\Sp_{2n}(\A_F)_{\mathfrak{S}_1}$ by the 2-cocycle $\prod_{v\in\mathfrak{S}}c_v(-,-)$.
In other words, we set
\begin{equation*}
\Mp_{2n}(\A_F)_\mathfrak{S}
=\Sp_{2n}(\A_F)_\mathfrak{S} \times \set{\pm1}
\end{equation*}
as sets, and define the multiplication by
\begin{equation*}
((g_v)_v, \epsilon)\cdot((g_v')_v, \epsilon')
=((g_vg_v')_v, \epsilon\epsilon'\prod_{v \in \mathfrak{S}}c_v(g_v,g_v')).
\end{equation*}
For two such finite sets $\mathfrak{S}_1\subset\mathfrak{S}_2$, we define the embedding $\Mp_{2n}(\A_F)_{\mathfrak{S}_1}\into\Mp_{2n}(\A_F)_{\mathfrak{S}_2}$ by
\begin{equation*}
((g_v)_v, \epsilon)
\mapsto ((g_v)_v, \epsilon\prod_{v \in \mathfrak{S}_2\setminus\mathfrak{S}_1}s_v(g_v)).
\end{equation*}
The global metaplectic group $\Mp_{2n}(\A_F)$ is defined by the inductive limit
\begin{equation*}
\Mp_{2n}(\A_F)
=\varinjlim_{\mathfrak{S}} \Mp_{2n}(\A_F)_\mathfrak{S},
\end{equation*}
where $\mathfrak{S}$ extends over all such finite sets of places of $F$.
This is a double cover of $\Sp_{2n}(\A_F)$:
\begin{equation*}
1 \lra \set{\pm1} \lra \Mp_{2n}(\A_F) \lra \Sp_{2n}(\A_F) \lra 1,
\end{equation*}
but note that in the expression $(g,\epsilon)\in\Mp_{2n}(\A_F)_\frakS$, the second component $\epsilon=\pm1$ depends on the choice of $\mathfrak{S}$.
Thus, we cannot identify $\Mp_{2n}(\A_F)$ with $\Sp_{2n}(\A_F)\times\{\pm1\}$ as sets.
The covering splits over $\Sp_{2n}(F)$, and the image of $\gamma \in \Sp_{2n}(F)$ is given by $(\gamma, 1) \in \Mp_{2n}(\A_F)_\mathfrak{S}$, for sufficiently large $\mathfrak{S}$ depending on $\gamma$.
In this paper, we regard $\Sp_{2n}(F)$ as a subgroup of $\Mp_{2n}(\A_F)$ in this way.\\

A representation of the metaplectic group over a local field or a ring of adeles is said to be genuine if it does not factor through the covering map.
A function on the metaplectic group is also said to be genuine if it satisfies the same condition.
Let $L^2(\Mp_{2n})$ be the subspace of $L^2(\Sp_{2n}(F) \backslash \Mp_{2n}(\A_F))$ consisting of the genuine functions, that is to say, the maximum subspace on which $\{\pm1\}$ acts as the nontrivial character.

\subsection{Elliptic $A$-parameters}
Let $F$ be a number field, $\A_F$ ring of adeles of $F$, and $\psi$ a nontrivial additive character of $F\backslash \A_F$.
The notion of elliptic $A$-parameters for $\SO_5$ and for $\Mp_4$ are the same.
Recall from \cite{art, gi18, gi20} that an elliptic $A$-parameter for $\SO_5$ or $\Mp_4$ is a formal unordered finite direct sum
\begin{align*}
\phi
=\bigoplus_i \phi_i \boxtimes S_{d_i},
\end{align*}
where
\begin{itemize}
\item each $\phi_i$ is an irreducible self-dual cuspidal automorphic representation of $\GL_{n_i}(\A_F)$;
\item $S_d$ denotes the unique irreducible $d$-dimensional representation of $\SL_2(\C)$;
\item if $d_i$ is odd, then $\phi_i$ is symplectic;
\item if $d_i$ is even, then $\phi_i$ is orthogonal;
\item if $(\phi_i, d_i)=(\phi_j, d_j)$, then $i=j$;
\item $\sum_i n_i d_i = 4$.
\end{itemize}
Moreover, if $d_i=1$ for all $i$, then we say that $\phi$ is tempered.

The global component group $S_\phi$ of an elliptic $A$-parameter $\phi=\bigoplus_i \phi_i \boxtimes S_{d_i}$ is defined as a free $\Z/2\Z$-module
\begin{align*}
S_\phi
=\bigoplus_i (\Z/2\Z) a_i
\end{align*}
with a formal basis $\{a_i\}$ such that $a_i$ corresponds to $\phi_i\boxtimes S_{d_i}$.
Recall also that Arthur \cite{art} associated a character $\epsilon_\phi$ of $S_\phi$ with $\phi$, and Gan-Ichino \cite{gi20} did another character $\wtil{\epsilon}_\phi$.
\subsection{Local $A$-packets}
For an elliptic $A$-parameter $\phi=\bigoplus_i \phi_i \boxtimes S_{d_i}$ for $\SO_5$ or $\Mp_4$ and a place $v$ of $F$, the localization $\phi_v : L_{F_v} \times \SL_2(\C) \to \Sp_4(\C)$ is defined by
\begin{align*}
\phi_v
= \bigoplus_i \phi_{i, v} \boxtimes S_{d_i}.
\end{align*}
Here, the irreducible representation $\phi_{i,v}$ of $\GL_{n_i}(F_v)$ is identified with an $L$-parameter $L_{F_v} \to \GL_{n_i}(\C)$ via the local Langlands correspondence \cite{lan, ht, hen}.
For such $\phi_v$, the associated $L$-parameter $\varphi_{\phi_v} : L_{F_v} \to \Sp_4(\C)$ is defined by
\begin{align*}
\varphi_{\phi_v} (w)
=\phi_v(w, \mmatrix{|w|_v^{\frac{1}{2}}}{}{}{|w|_v^{-\frac{1}{2}}}).
\end{align*}
We let $S_{\phi_v}$ and $S_{\varphi_{\phi_v}}$ denote the component groups of the centralizers $\cent(\image(\phi_v), \Sp_4(\C))$ and $\cent(\image(\varphi_{\phi_v}), \Sp_4(\C))$, respectively.
Note that there exists a natural map $S_\phi \to S_{\phi_v}$.

Let us write $\Pi_{\phi_v}(\SO_5)$ and $\Pi_{\phi_v, \psi_v}(\Mp_4)$ for the local $A$-packets associated to $\phi_v$.
Recall that $\Pi_{\phi_v}(\SO_5)$, given by Arthur \cite{art}, is a finite set
\begin{align*}
\Pi_{\phi_v}(\SO_5)
=\Set{\sigma_{\eta_v} | \eta_v \in \what{S}_{\phi_v}}
\end{align*}
of semisimple representations of $\SO_5(F_v)$ of finite length indexed by characters of $S_{\phi_v}$, and $\Pi_{\phi_v, \psi_v}(\Mp_4)$, given by Gan-Ichino \cite{gi18, gi20}, is a finite set
\begin{align*}
\Pi_{\phi_v, \psi_v}(\Mp_4)
=\Set{\pi_{\eta_v}=\pi_{\eta_v, \psi_v} | \eta_v \in \what{S}_{\phi_v}}
\end{align*}
of semisimple genuine representations of $\Mp_4(F_v)$ of finite length indexed by characters of $S_{\phi_v}$.
Note that the latter packet $\Pi_{\phi_v, \psi_v}(\Mp_4)$ depends on the additive character $\psi_v$ of $F_v$.
Recall also that
\begin{itemize}
\item if $v$ is a finite place and $\sigma_{\eta_v}$ (resp. $\pi_{\eta_v', \psi_v}$) has an unramified constituent, then $\phi_{i,v}$ are unramified and $\eta_v$ (resp. $\eta_v'$) is a trivial character;
\item if $\phi$ is tempered, then the local $A$-packets $\Pi_{\phi_v}(\SO_5)$ and $\Pi_{\phi_v, \psi_v}(\Mp_4)$ coincide with $L$-packets associated to the $L$-parameter $\varphi_{\phi_v}=\phi_v$ given by \cite{art, ab95, ab98, gs}.
\end{itemize}

\subsection{Multiplicity formulae}
Let us write $L^2_{\mathrm{disc}}(\Mp_4)$ for the discrete spectrum of the genuine unitary representation $L^2(\Mp_4)$ of $\Mp_4(\A_F)$, and $L^2_{\mathrm{disc}}(\SO_5)$ for the discrete spectrum of the unitary representation $L^2(\SO_5(F) \backslash \SO_5(\A_F))$ of $\SO_5(\A_F)$.
The multiplicity formulae describe the decomposition of $L^2_{\mathrm{disc}}(\Mp_4)$ and $L^2_{\mathrm{disc}}(\SO_5)$ into near equivalence classes of irreducible representations, and the multiplicity of any irreducible representation in each near equivalence classes.
Here we say that two irreducible representations $\pi=\otimes_v\pi_v$ and $\pi'=\otimes_v\pi_v'$ of $\Mp_4(\A_F)$ or $\SO_5(\A_F)$ are nearly equivalent if $\pi_v$ and $\pi_v'$ are equivalent for almost all places $v$ of $F$.

The adelic component group $S_{\phi, \A}$ of an elliptic $A$-parameter $\phi$ for $\SO_5$ or $\Mp_4$ is defined by the infinite direct product
\begin{align*}
S_{\phi, \A}
=\prod_v S_{\phi_v},
\end{align*}
and we shall write $\Delta$ for the diagonal map $S_\phi \to S_{\phi, \A}$.
Let $\what{S}_{\phi, \A}$ be the group of continuous characters of $S_{\phi, \A}$.
Note that $\what{S}_{\phi, \A}= \bigoplus_v \what{S}_{\phi_v}$.
For any $\eta=\bigotimes_v \eta_v \in \what{S}_{\phi, \A}$, let $\sigma_\eta$ and $\pi_\eta=\pi_{\eta, \psi}$ denote semisimple representations
\begin{align*}
\sigma_\eta &= \bigotimes_v \sigma_{\eta_v},&
\pi_{\eta, \psi} &= \bigotimes_v \pi_{\eta_v, \psi_v}
\end{align*}
of $\SO_5(\A_F)$ and $\Mp_4(\A_F)$, respectively.
Then the multiplicity formulae for $\SO_5$ and $\Mp_4$ are stated as follows.
\begin{thm}[{\cite[Theorem 1.5.2]{art}}]\label{amf}
For every elliptic $A$-parameter $\phi$ for $\SO_5$, put
\begin{align*}
L^2_\phi(\SO_5)&=\bigoplus_{\eta \in \what{S}_{\phi, \A}} n_\eta \sigma_\eta,\\
n_\eta&=
\begin{cases}
1, & \text{if $\eta \circ \Delta = \epsilon_\phi$},\\
0, & \text{otherwise}.
\end{cases}
\end{align*}
Then each $L^2_\phi(\SO_5)$ is a full near equivalence class of irreducible representations in the discrete spectrum $L^2_{\mathrm{disc}}(\SO_5)$, and $L^2_{\mathrm{disc}}(\SO_5)$ can be decomposed into a direct sum
\begin{align*}
L^2_{\mathrm{disc}}(\SO_5)
=\bigoplus_\phi L^2_\phi(\SO_5).
\end{align*}
\end{thm}

\begin{thm}[{\cite[Theorem 2.1]{gi20}}]\label{gimf}
For every elliptic $A$-parameter $\phi$ for $\Mp_4$, put
\begin{align*}
L^2_{\phi,\psi}(\Mp_4)&=\bigoplus_{\eta \in \what{S}_{\phi, \A}} m_\eta \pi_\eta,\\
m_\eta&=
\begin{cases}
1, & \text{if $\eta \circ \Delta = \wtil{\epsilon}_\phi$},\\
0, & \text{otherwise}.
\end{cases}
\end{align*}
Then each $L^2_{\phi,\psi}(\Mp_4)$ is a full near equivalence class of irreducible representations in the discrete spectrum $L^2_{\mathrm{disc}}(\Mp_4)$, and $L^2_{\mathrm{disc}}(\Mp_4)$ can be decomposed into a direct sum
\begin{align*}
L^2_{\mathrm{disc}}(\Mp_4)
=\bigoplus_\phi L^2_{\phi,\psi}(\Mp_4).
\end{align*}
\end{thm}

\section{Jacobi forms}\label{Jacobi form}
In this section, we shall recall the general definitions of holomorphic and skew-holomorphic Jacobi forms, and a correspondence between half-integral weight modular forms and holomorphic or skew-holomorphic Jacobi forms of degree 2.
\subsection{Holomorphic and skew holomorphic Jacobi forms of index 1}
For any commutative ring $R$ with unity 1, let $\calH_n(R)$ be the Heisenberg group of degree $n$, i.e.,
\begin{align*}
\calH_n(R)
=\Set{([\lambda, \mu], \kappa) | \lambda, \mu \in R^n, \ \kappa \in R},
\end{align*}
with
\begin{align*}
([\lambda_1, \mu_1], \kappa_1) \cdot ([\lambda_2, \mu_2], \kappa_2)
=([\lambda_1 +\lambda_2, \mu_1+\mu_2], \kappa_1+\kappa_2+\tp{\lambda_1}\mu_2-\tp{\mu_1}\lambda_2).
\end{align*}
By using embeddings
\begin{align*}
\Sp_{2n}(R) &\into \Sp_{2(n+1)}(R),&
\mmatrix{A}{B}{C}{D} &\mapsto
\left(\begin{array}{cccc}
1&&&\\
&A&&B\\
&&1&\\
&C&&D\end{array}\right),&
\\
\calH_n(R) &\into \Sp_{2(n+1)}(R),&
([\lambda, \mu], \kappa) &\mapsto
\left(\begin{array}{cccc}
1&\tp{\lambda}&\kappa&\tp{\mu}\\
&1_n&\mu&\\
&&1&\\
&&-\lambda&1_n\end{array}\right),&
\end{align*}
we can regard the Jacobi group over $R$ of degree $n$ as a subgroup $\Sp^J_{2n}(R)=\Sp_{2n}(R) \ltimes \calH_n(R)$ of $\Sp_{2(n+1)}(R)$.
The center of $\Sp^J_{2n}(R)$ is $\{([0,0], \kappa)| \kappa \in R\}$.
We shall write $\calZ(R)$ for this subgroup.
For future references we note an equation here:
\begin{equation*}
\left(\begin{array}{cc}
A&B\\C&D
\end{array}\right)
([\lambda,\mu],\kappa)
\left(\begin{array}{cc}
A&B\\C&D
\end{array}\right)\inv
=
([-C\mu+D\lambda,A\mu-B\lambda],\kappa).
\end{equation*}

Let $(\rho, V)$ be an irreducible finite dimensional polynomial representation of $\GL_n(\C)$.
For any function $F : \frakH_n\times\C^n \to V$, we define a group action of $\Sp_{2n}^J(\R)$ (of index 1):
\begin{align*}
\left[F|^\mathrm{hol}_{\rho, 1} ([\lambda,\mu], \kappa)\right](Z,w)
&=\bfe(\tp{\lambda}Z\lambda + 2 \tp{\lambda}w +\tp{\lambda}\mu + \kappa)F(Z,w+Z\lambda+\mu),\\
\left[F|^\mathrm{hol}_{\rho, 1} g\right](Z,w)
&=\bfe(-\tp{w} J(g,Z)\inv Cw) \rho(J(g,Z))\inv F(gZ, \tp{J(g,Z)}\inv w),
\end{align*}
where $g=\mmatrix{*}{*}{C}{*}\in\Sp_{2n}(\R)$.

\begin{defn}
Let $F(Z,w)$ be a $V$-valued function on $(Z,w)\in \frakH_n \times \C^n$.
When $F$ satisfies the next conditions (0)-(2), we say that $F$ is a holomorphic Jacobi form of weight $\rho$ with index 1.
\begin{enumerate}[(1)]
\setcounter{enumi}{-1}
\item $F$ is holomorphic on $\frakH_n\times\C^n$;
\item $F|^\mathrm{hol}_{\rho, 1} g =F$ for all $g \in \Sp^J_{2n}(\Z)$;
\item $F$ has a Fourier expansion of the following form:
\begin{align*}
F(Z,w)
=\sum_{\substack{(N,r) \in L_n^* \times \Z^n \\ 4N-r\tp{r} \geq 0}} A(N,r) \bfe(\tr(NZ)+\tp{r}w),
\end{align*}
where $L_n^*$ denotes the set of all half-integral symmetric matrices.
\end{enumerate}
Moreover, if $F$ satisfies the next condition, we say that $F$ is a holomorphic Jacobi cusp form.
\begin{enumerate}[(1)]
\setcounter{enumi}{2}
\item $A(N,r)=0$ unless $4N-r\tp{r} >0$.
\end{enumerate}
The space of holomorphic Jacobi forms (resp. holomorphic Jacobi cusp forms) of weight $\rho$ with index 1 is denoted by $J^\mathrm{hol}_{\rho,1}$ (resp. $J_{\rho,1}^{\mathrm{hol, cusp}}$).
\end{defn}

Next, let us recall the definition of skew-holomorphic Jacobi forms.
For any function $F : \frakH_n\times\C^n \to V$, we define a group action of $\Sp^J_{2n}(\R)$ (of index 1):
\begin{align*}
F|_{\rho, 1}^\mathrm{skew} ([\lambda,\mu], \kappa)
&=F|^\mathrm{hol}_{\rho, 1} ([\lambda,\mu], \kappa),\\
[F|_{\rho, 1}^\mathrm{skew} g](Z,w)
&=\bfe(-\tp{w} J(g,Z) C w) \frac{|\det J(g,Z)|}{\det J(g,Z)} \overline{\rho(J(g,Z))}\inv F(gZ, \tp{J(g,Z)}\inv w),
\end{align*}
where $g=\mmatrix{*}{*}{C}{*} \in \Sp_{2n}(\R)$.

\begin{defn}
Let $F(Z,w)$ be a $V$-valued function on $(Z,w) \in \frakH_n \times \C^n$.
When $F$ satisfies the next conditions (0)-(2), we say that $F$ is a skew-holomorphic Jacobi form of weight $\rho$ with index 1.
\begin{enumerate}[(1)]
\setcounter{enumi}{-1}
\item $F$ is holomorphic in $w \in \C^n$ and real analytic in the real part and imaginary part of $Z \in \frakH_n$;
\item $F|_{\rho, 1}^\mathrm{skew} g = F$ for all $g \in \Sp^J_{2n}(\Z)$;
\item $F$ has a Fourier expansion of the following form:
\begin{align*}
F(Z,w)
=\sum_{\substack{(N, r) \in L_n^* \times \Z^n \\ 4N-r\tp{r} \leq 0}}
 A(N,r) \bfe(\tr(NZ-\frac{1}{2}i(4N-r\tp{r})Y)) \bfe(\tp{r}w),
\end{align*}
where $Y$ is the imaginary part of $Z$.
\end{enumerate}
Moreover, if $F$ satisfies the next condition, we say that $F$ is a skew-holomorphic Jacobi cusp form.
\begin{enumerate}[(1)]
\setcounter{enumi}{2}
\item $A(N,r)=0$ unless $4N-r\tp{r}<0$.
\end{enumerate}
The space of skew-holomorphic Jacobi forms (resp. skew-holomorphic Jacobi cusp forms) of weight $\rho$ with index 1 is denoted by $J_{\rho,1}^\mathrm{skew}$ (resp. $J_{\rho,1}^\mathrm{skew, cusp}$).
\end{defn}

\subsection{Hecke operators and the relation to modular forms of half-integral weight}
In this subsection, we define Hecke operators $T^J_s(p)$ ($s=0,1,2$) of holomorphic or skew-holomorphic Jacobi forms of degree 2.
In this case, $(\rho, V)$ is an irreducible finite dimensional polynomial representation of $\GL_2(\C)$, and any such representation is isomorphic to $(\det^k \otimes \Sym_j, V_j)$ for some $k$ and $j$.
Let us write $(k,j)$ for the representation $(\det^k \otimes \Sym_j, V_j)$.\\

To define Hecke operators, let us first extend the slash operators $|^\mathrm{hol}$ and $|^\mathrm{skew}$ from $\Sp_4(\R)$ to $\GSp_4^+(\R)$ so that the actions of scalar matrices are trivial, i.e.,
\begin{align*}
\left[F|^\mathrm{hol}_{(k,j),1}g\right] (Z,w)
&=\nu(g)^{k+\frac{j}{2}} \bfe(-\tp{w} J(g,Z)\inv Cw)\\
&\quad \times \det(J(g,Z))^{-k} \Sym_j(J(g,Z))\inv F(gZ, \nu(g)^\frac{1}{2} \tp{J(g,Z)}\inv w),\\
\left[F|^\mathrm{skew}_{(k,j),1}g\right] (Z,w)
&=\nu(g)^{k+\frac{j}{2}} \bfe(-\tp{w} J(g,Z)\inv Cw)\\
& \quad \times \frac{|\det J(g,Z)|}{\det J(g,Z)} \overline{\det J(g,Z)^{-k} \Sym_j(J(g,Z))\inv}
 F(gZ, \nu(g)^\frac{1}{2} \tp{J(g,Z)}\inv w),
\end{align*}
for $g \in \GSp_4^+(\R)$ and $F : \frakH_2 \times \C^2 \to V_j$.\\

Let $s=0,1,2$.
Then for any prime number $p$, the Hecke operator $T^J_s(p)$ is defined by
\begin{align*}
\Sp_4(\Z) K_s(p^2) \Sp_4(\Z)
=\bigsqcup_t \Sp_4(\Z) g_{s,t},
\end{align*}
and
\begin{align*}
F|_{(k,j),1}^\star T^J_s(p)
=\sum_{\lambda, \mu \in (\Z/p\Z)^2} \sum_t F|_{(k,j),1}^\star g_{s,t} |_{(k,j),1}^\star([\lambda,\mu], 0),
\end{align*}
for $F \in J_{(k,j),1}^\star$, where $\star \in \{ \mathrm{hol}, \mathrm{skew} \}$.
One can also define Hecke operators in the higher degree case in the similar way.
Here we remark that $g_{s,t}$ is the $\Sp_4(\R)$ component of $\wtil{g}_{s,t}$ if $p\neq2$.

As explained in \cite{ibuconj}, there exists a canonical isomorphism between $J_{(k,j),1}^\mathrm{hol, cusp}$ or $J_{(k,j),1}^\mathrm{skew, cusp}$ and $S_{k-\frac{1}{2}, j}^+(\Gamma_0(4), \left(\frac{-1}{\cdot}\right)^l)$.
\begin{thm}\label{5.1}
There exist linear isomorphisms
\begin{align*}
J_{(k,j),1}^\mathrm{hol, cusp}
&\overset{\simeq}{\lra} S_{k-\frac{1}{2},j}^+(\Gamma_0(4), \left(\frac{-1}{\cdot}\right)^k),\\
J_{(k,j),1}^\mathrm{skew, cusp}
&\overset{\simeq}{\lra} S_{k-\frac{1}{2},j}^+(\Gamma_0(4), \left(\frac{-1}{\cdot}\right)^{k-1}),
\end{align*}
for which we shall write $\Psi$, such that
\begin{align*}
\Psi\left( F|_{(k,j),1}^\star T^J_s(p) \right)
=p^{3+\frac{s}{2}} \left(\frac{-1}{p}\right)^{(k+\delta)s} \Psi(F)|_{k-\frac{1}{2}, j} T_s(p),
\end{align*}
for any odd prime $p$ and $F \in J_{(k,j),1}^{\star, \mathrm{cusp}}$, where $\delta=0$ if $\star=\mathrm{hol}$, and $\delta=1$ if $\star=\mathrm{skew}$.
\end{thm}
The isomorphism $\Psi$ is defined explicitly.
See \cite[Theorem5.1]{ibuconj} for detail.

Now we state the definition of the Hecke operator $T_s(2)$ on $S_{(k-\frac{1}{2},j)}^+(\Gamma_0(4), \left(\frac{-1}{\cdot}\right)^l)$ for $p=2$ by
\begin{equation*}
\Psi \left( F|_{(k,j),1}^\star T^J_s(2) \right)
=2^{3+\frac{s}{2}} \Psi(F)|_{k-\frac{1}{2}, j} T_s(2),
\end{equation*}
for $F \in J_{(k,j),1}^{\star, \mathrm{cusp}}$ and $\star \in \{ \mathrm{hol}, \mathrm{skew} \}$.

\section{Representation theory of the Jacobi groups}\label{Jacobi group}
In this section, following \cite{bs}, we recall the general theory of representations of the Jacobi groups.
Although \cite{bs} treats only the case of $n=1$, a lot of arguments in the book also go when $n\geq1$.
\subsection{The $p$-adic case}
First we shall consider representations of the Jacobi group over a $p$-adic field.
Let $F$ be a $p$-adic field, i.e., $F$ is a finite extension of $\Q_p$ for some $p$.
Let $\varpi \in F$ be a uniformizer, $\calO$ the maximal compact subring of $F$, $q=q_F$ the number of elements in the residual field $\calO/(\varpi)$,  $|-|_F$ the normalized absolute value on $F$ so that $|\varpi|_F=q\inv$.
Let $\psi : F\to\C^1$ be a nontrivial additive character of order zero, i.e., $\psi$ is trivial on $\calO$ and nontrivial on $\varpi\inv \calO$.
The Schr\"odinger representation $\pi_{S,\psi}$ of $\calH_n(F)$ on the Schwartz space $\calS(F^n)$ is defined by
\begin{equation*}
\left[\pi_{S,\psi}(([\lambda, \mu], \kappa)) f \right](x)
=\psi(\kappa+\tp(2x+\lambda)\mu) f(x+\lambda),
\quad
f \in \calS(F^n), \ ([\lambda,\mu],\kappa)\in\calH_n(F).
\end{equation*}
The following fact is known as the Stone-von Neumann theorem.
\begin{thm}
\begin{enumerate}[(1)]
\item The Schr\"odinger representation $\pi_{S,\psi}$ is the unique irreducible smooth representation of $\calH_n(F)$ with central character $\psi$.
\item Any smooth representation of $\calH_n(F)$ with central character $\psi$ is isomorphic to a direct sum of $\pi_{S,\psi}$.
\end{enumerate}
\end{thm}
The Stone-von Neumann theorem gives us the Weil representation $\omega_\psi$ of the metaplectic group $\Mp_{2n}(F)$ on $\calS(F^n)$.
Combining the Schr\"odinger representation $\pi_{S,\psi}$ and the Weil representation $\omega_\psi$, we obtain the Schr\"odinger-Weil representation $\pi_{\SW,\psi}$ of $\Mp^J_{2n}(F)=\Mp_{2n}(F)\ltimes \calH_n(F)$, the metaplectic double covering group of $\Sp^J_{2n}(F)$.
Note that the Schr\"odinger-Weil representation $\pi_{\SW,\psi}$ is a genuine representation, i.e., it does not factor through $\Mp^J_{2n}(F)\to\Sp^J_{2n}(F)$.

If $\pi'$ is a genuine representation of $\Mp_{2n}(F)$, then a representation $\pi=\pi' \otimes \pi_{\SW,\psi}$ is not genuine and can be regarded as a representation of $\Sp^J_{2n}(F)$.
Also, we have the following fact.
(The proof is the same as \cite[Theorem 2.6.2 and Proposition 5.1.2]{bs}.)
\begin{thm}\label{ccp}
The correspondence $\pi' \mapsto \pi:=\pi' \otimes \pi_{\SW,\psi}$ gives a bijection between the irreducible genuine smooth representations of $\Mp_{2n}(F)$ and the irreducible smooth representations of $\Sp^J_{2n}(F)$ with central character $\psi$.
Moreover, $\pi$ is admissible if and only if $\pi'$ is admissible.
\end{thm}

Next, we shall consider principal series representations.
Let $T$ be the standard maximal torus in $\Sp_{2n}$ consisting of the diagonal matrices, and $B$ the standard Borel subgroup of $\Sp_{2n}(F)$ consisting of matrices of the form $\mmatrix{A}{AS}{0}{\tp{A}\inv}$, with $A\in\GL_n(F)$ upper triangular and $S$ symmetric.
Let $B^J$ be a subgroup of $\Sp^J_{2n}(F)$ generated by $B$ and $\{ ([0,\mu], \kappa) | \mu \in F^n, \ \kappa\in F \}$.
For characters $\chi_1, \ldots, \chi_n : F^\times \to \C^\times$, following \cite{gs}, we write $I_{B,\psi}(\chi_1,\ldots, \chi_n)$ for the principal series representation of $\Mp_{2n}(F)$.
Let $I_{B^J}(\chi_1, \ldots, \chi_n; \psi)$ be a space of functions $\Phi : \Sp^J_{2n}(F) \to \C$ satisfying the following conditions:
\begin{itemize}
\item there exists an open subgroup $H \subset \Sp^J_{2n}(F)$ such that $\Phi$ is right $H$-invariant;
\item for any $g\in \Sp^J_{2n}(F)$, $\kappa\in F$, $\mu \in F^n$, a symmetric matrix $S\in \operatorname{M}_{n\times n}(F)$, and an upper triangular matrix $A \in \GL_n(F)$ with diagonal elements $(a_1, \ldots, a_n)$,
\begin{align*}
\Phi\left( \left(\begin{array}{cc} A&AS\\0&\tp{A}\inv \end{array}\right) ([0,\mu],\kappa) g \right)
=\prod_{i=1}^n |a_i|_F^{n+\frac{3}{2}-i} \chi_i(a_i) \cdot \psi(\kappa) \Phi(g).
\end{align*}
\end{itemize}
The Jacobi group $\Sp^J_{2n}(F)$ acts on $I_{B^J}(\chi_1, \ldots, \chi_n; \psi)$ by right translation.
This is called a principal series representation with central character $\psi$, and we have the following theorem.
(The proof is the same as \cite[Theorem 5.4.2]{bs}.)
\begin{thm}\label{ps}
The principal series representations of $\Sp^J_{2n}$ and $\Mp_{2n}$ are corresponding in a canonical way:
\begin{align*}
I_{B, \overline{\psi}}(\chi_1, \ldots, \chi_n) \otimes \pi_{\SW, \psi}
&\cong I_{B^J}(\chi_1,\ldots, \chi_n; \psi),\\
\varphi \otimes f 
&\mapsto \Phi_{\varphi \otimes f},
\end{align*}
where $\Phi_{\varphi \otimes f}(gh) = \varphi(g) [\pi_{\SW, \psi}(gh) f](0)$  for $g \in \Mp_{2n}(F)$ and $h \in \calH_n(F)$.
\end{thm}

Assume that $\chi_i$ ($i=1,\ldots,n$) are unramified, and put $\alpha_i=\chi_i(\varpi) \in \C^\times$.
Then the representations $I_{B^J}(\chi_1, \ldots, \chi_n; \psi)$ and $I_{B,\psi}(\chi_1,\ldots, \chi_n)$ are called unramified principal series representations.
A representation of $\Sp^J_{2n}(F)$ is called $\Sp^J_{2n}(\calO)$-spherical if it contains a nonzero vector fixed by $\Sp^J_{2n}(\calO)$.
By a straightforward calculation, one can see that $I_{B^J}(\chi_1, \ldots, \chi_n; \psi)$ has a $\Sp^J_{2n}(\calO)$-fixed nonzero vector $\Phi_0$ such that $\Phi_0([\lambda,0],0)=1_{\calO^n}(\lambda)$, and any $\Sp^J_{2n}(\calO)$-fixed vector is a scalar multiple of $\Phi_0$.
Since any principal series representation of $\Mp_{2n}(F)$ has finite length, Theorems \ref{ccp} and \ref{ps} imply that $I_{B^J}(\chi_1, \ldots, \chi_n; \psi)$ has a unique irreducible $\Sp^J_{2n}(\calO)$-spherical constituent $\pi(\alpha_1,\ldots, \alpha_n;\psi)$ with central character $\psi$ coming from $\Phi_0$.
We shall call such a representation an unramified representation.
\begin{thm}\label{csu}
Any irreducible $\Sp^J_{2n}(\calO)$-spherical representation $\pi$ of $\Sp^J_{2n}(F)$ with central character $\psi$ is isomorphic to $\pi(\alpha_1,\ldots, \alpha_n; \psi)$ for some $\alpha_1, \ldots, \alpha_n \in \C^\times$, and $\pi(\alpha_1,\ldots, \alpha_n; \psi)\cong\pi(\alpha_1'\ldots,\alpha_n'; \psi)$ if and only if there exist $\sigma \in \frakS_n$ and $\varepsilon_1, \ldots,\varepsilon_n \in \{\pm1\}$ such that $\alpha_i'=\alpha_{\sigma(i)}^{\varepsilon_i}$.
\end{thm}
\begin{proof}
Following \cite[Definition 6.1.2]{bs}, we define the Hecke algebra $\calH(\Sp^J_{2n}(F), \Sp^J_{2n}(\calO))_\psi$ with character $\psi$ to be the space of functions $f : \Sp^J_{2n}(F) \to \C$ which are compactly supported modulo $\calZ(F)$, and satisfy
\begin{equation*}
f(kgk'z)=\overline{\psi(\kappa)} f(g),
\end{equation*}
for any $g \in \Sp^J_{2n}(F)$, $k, k' \in \Sp^J_{2n}(\calO)$, and $z=([0,0],\kappa) \in \calZ(F)$.
The product is defined by the convolution product
\begin{equation*}
f_1*f_2(x)
=\int_{\Sp^J_{2n}(F)/\calZ(F)} f_1(xy\inv)f_2(y) dy.
\end{equation*}
Then, as in \cite[(6.3)]{bs}, we have an injective mapping
\begin{align}\label{satake1}
\left\{
\begin{tabular}{l}
\text{Irreducible admissible}\\
\text{$\Sp^J_{2n}(\calO)$-spherical representations}\\
\text{of $\Sp^J_{2n}(F)$ with central character $\psi$}
\end{tabular}
\right\}_{/\simeq}
\into
\left\{
\begin{tabular}{l}
\text{Irreducible finite dimensional}\\
\text{smooth representations}\\
\text{of $\calH(\Sp^J_{2n}(F), \Sp^J_{2n}(\calO))_\psi$}
\end{tabular}
\right\}_{/\simeq}.
\end{align}
By \cite[\S4]{mu}, we have a canonical isomorphism $\calH(\Sp^J_{2n}(F), \Sp^J_{2n}(\calO))_\psi \cong \C[X_1^{\pm1}, \ldots, X_n^{\pm1}]^{\frakS_n\ltimes\{\pm1\}^n}$, and all irreducible finite dimensional smooth representations of $\calH(\Sp^J_{2n}(F), \Sp^J_{2n}(\calO))_\psi$ are one dimensional.
Therefore the right hand side of \eqref{satake1} can be regarded as 
\begin{align}\label{satake2}
\Hom_\C(\calH(\Sp^J_{2n}(F), \Sp^J_{2n}(\calO))_\psi, \C)
\cong \Hom_\C(\C[X_1^{\pm1}, \ldots, X_n^{\pm1}]^{\frakS_n \ltimes \{\pm1\}^n}, \C)
&\cong (\C^\times)^n/ \frakS_n\ltimes\{\pm1\}^n,\\
(X_i \mapsto \alpha_i)_i
&\mapsto (\alpha_i)_i.\notag
\end{align}
The composition of \eqref{satake1} and \eqref{satake2} sends $\pi(\alpha_1, \ldots, \alpha_n; \psi)$ to $(\alpha_1,\ldots, \alpha_n)$.
Now the assertion follows.
\end{proof}

Note that the mapping \eqref{satake1} is bijective.
We shall call $(\alpha_1, \ldots, \alpha_n; \psi)$ or $(\alpha_1, \ldots, \alpha_n)$ the Satake parameter of $\pi(\alpha_1, \ldots, \alpha_n; \psi)$.\\

On the other hand, we shall call an irreducible genuine representation of $\Mp_{2n}(F)$ corresponding to an unramified $L$-parameter via the LLC \cite{gs} with respect to $\psi$ an irreducible $\psi$-unramified representation, or simply an unramified representation.
Let again $\chi_i$ ($i=1,\ldots,n$) be unramified characters of $F^\times$, and put $\alpha_i=\chi_i(\varpi)\in\C^\times$.
Then we shall write $\pi_\psi(\alpha_1,\ldots,\alpha_n)$ for the irreducible $\psi$-unramified representation of $\Mp_{2n}(F)$ with the $L$-parameter $\chi_1\oplus\cdots\oplus\chi_n\oplus\chi_1^\vee\oplus\cdots\oplus\chi_n^\vee$ with respect to $\psi$, where $\chi_i$ is identified with a character of the Weil group of $F$ through the local class field theory.
The tuple $(\alpha_1,\ldots,\alpha_n) \in \C^n/\frakS_n\ltimes\{\pm1\}^n$ is called its Satake parameter.

If $p\neq2$, the metaplectic covering $\Mp_{2n}(F)\to\Sp_{2n}(F)$ splits uniquely over $\Sp_{2n}(\calO)$, and a representation of $\Mp_{2n}(F)$ is called $\Sp_{2n}(\calO)$-spherical if it contains a nonzero vector fixed by $\Sp_{2n}(\calO)$.
Since $\psi$ has order zero, when $p\neq2$, it is known that an irreducible genuine representation of $\Mp_{2n}(F)$ is unramified if and only if it is $\Sp_{2n}(\calO)$-spherical.
Moreover, by the higher degree generalization of \cite[Proposition 6.3.5]{bs} it can be seen that Theorem \ref{ccp} associates a spherical representation $\pi_{\overline{\psi}}(\alpha_1, \ldots, \alpha_n)$ of $\Mp_{2n}(F)$ with a spherical representation $\pi(\alpha_1,\ldots, \alpha_n; \psi)$ of $\Sp^J_{2n}(F)$.
When $p=2$, however, there is no notion of spherical representations of $\Mp_{2n}(F)$.
Nevertheless, Theorem \ref{ccp} associates an unramified representation of $\Mp_{2n}(F)$ with that of $\Sp^J_{2n}(F)$:
\begin{thm}\label{ccu}
Let the residue characteristic $p$ of $F$ be any prime number including 2.
For any $(\alpha_1, \ldots, \alpha_n) \in (\C^\times)^n/\frakS_n\ltimes\{\pm1\}^n$, we have
\begin{equation*}
\pi(\alpha_1,\ldots, \alpha_n; \psi)
\cong \pi_{\overline{\psi}}(\alpha_1, \ldots, \alpha_n) \otimes \pi_{\SW,\psi}.
\end{equation*}
\end{thm}
\begin{rem}
If $p$ is odd, this has been already seen, but the following argument goes regardless of the parity of $p$.
\end{rem}
\begin{rem}
Our argument is completely different from that of \cite{ssj}.
\end{rem}
\begin{proof}
Let $\chi_i$ be the unramified character of $F^\times$ such that $\chi_i(\varpi)=\alpha_i$, and $s_i$ a real number such that $|\chi_i|=|-|_F^{s_i}$.
We may assume $s_1 \geq \cdots \geq s_n \geq 0$.
Thanks to the preceding theorems, it suffices to show that $\pi_{\overline{\psi}}(\alpha_1, \ldots, \alpha_n) \otimes \pi_{\SW,\psi}$ has a nonzero vector fixed by $\Sp^J_{2n(F)}$.\\

First, we treat the case of $n=1$.
In this case, put $s=s_1$, $\chi=\chi_1$, and $\alpha=\alpha_1$.
Let us divide the problem into two cases: I) $s=0$ and II) $s>0$.

I) $s=0$. Then the unramified principal series representation $I_{B, \overline{\psi}}(\chi)$ is irreducible, so the assertion follows from Theorem \ref{ps}.

II)$s>0$. Then $|\chi(\varpi)|<1$.
It is known that the unramified representation $\pi_{\overline{\psi}}(\alpha)$ is the image of an intertwining operator $\calT$ from $I_{B, \overline{\psi}}(\chi)$ to $I_{B, \overline{\psi}}(\chi^\vee)$ that is defined by
\begin{align*}
(\calT \varphi)(g)
=\int_F \varphi ((\mmatrix{0}{1}{-1}{0}\mmatrix{1}{x}{0}{1}, 1) g) dx,
\quad g \in \Mp_2(F).
\end{align*}
Note that since $s>0$ the integral converges absolutely and the intertwining operator $\calT$ is well-defined and nonzero.
As in \cite[p.123]{bs}, we obtain an intertwining operator
\begin{align*}
\calT^J = \calT \otimes 1
:I_{B^J}(\chi; \psi) \lra I_{B^J}(\chi^\vee; \psi).
\end{align*}
Then the image of $\calT^J$ is $\pi_{\overline{\psi}}(\alpha) \otimes \pi_{\SW, \psi}$.
Hence, it suffices to show that $\calT^J \Phi_0 \neq 0$, where $\Phi_0$ is the $\Sp^J_2(\calO)$-fixed nonzero vector in $I_{B^J}(\chi; \psi)$ such that $\Phi_0([\lambda,0],0)=1_{\calO}(\lambda)$.

By \cite[Proposition 5.6.3]{bs}, $(\calT^J \Phi_0)(1)$ is equal, up to a nonzero scalar multiple, to
\begin{equation*}
\int_F \int_F \Phi_0(\mmatrix{0}{1}{-1}{0}\mmatrix{1}{x}{0}{1} ([0,\mu], 0)) d\mu dx,
\end{equation*}
which is convergent, but not absolutely convergent.
In particular, note that the order of integration must not be changed.
Now we shall calculate it carefully.
\begin{align}
&\int_F \int_F \Phi_0(\mmatrix{0}{1}{-1}{0}\mmatrix{1}{x}{0}{1} ([0,\mu], 0)) d\mu dx \notag \\
&=\int_F \left\{
\Phi_0(\mmatrix{0}{1}{-1}{0}\mmatrix{1}{x}{0}{1})
+ \sum_{m=1}^\infty \int_{\calO^\times} \Phi_0(\mmatrix{0}{1}{-1}{0}\mmatrix{1}{x}{0}{1} ([0, \varpi^{-m}u], 0)) q^m du
\right\} dx \notag \\
&=
\int_{\calO} \left\{
\Phi_0(\mmatrix{0}{1}{-1}{0} \mmatrix{1}{x}{0}{1})
+ \sum_{m=1}^\infty \int_{\calO^\times} \Phi_0(\mmatrix{0}{1}{-1}{0} \mmatrix{1}{x}{0}{1} ([0, \varpi^{-m}u], 0)) q^m du
\right\} dx \label{a} \\
&\quad+
\sum_{l=1}^\infty \int_{\calO^\times}\left\{
\Phi_0(\mmatrix{0}{1}{-1}{0}\mmatrix{1}{\varpi^{-l}v}{0}{1})
+ \sum_{m=1}^\infty \int_{\calO^\times} \Phi_0(\mmatrix{0}{1}{-1}{0}\mmatrix{1}{\varpi^{-l}v}{0}{1} ([0, \varpi^{-m}u], 0)) q^m du
\right\} q^l dv. \label{b}
\end{align}
By  the definition of $\Phi_0$, the first term is
\begin{align*}
\eqref{a}
=\Phi_0(1) + \sum_{m=1}^\infty q^m \int_{\calO^\times} \Phi_0(([\varpi^{-m}u,0], 0)) du
=\Phi_0(1)
=1.
\end{align*}
Since we have
\begin{align}\label{a1}
\mmatrix{0}{1}{-1}{0} \mmatrix{1}{\varpi^{-l}v}{0}{1}
=\mmatrix{-\varpi^l v\inv}{1}{0}{-\varpi^{-l}v} \mmatrix{1}{0}{\varpi^l v\inv}{1}
\end{align}
and
\begin{align*}
\mmatrix{1}{0}{\varpi^l v\inv}{1} ([0, \varpi^{-m}u], 0)
=([0, \varpi^{-m}u], -\varpi^{l-2m}v\inv u^2) ([-\varpi^{l-m}v\inv u, 0], 0) \mmatrix{1}{0}{\varpi^l v\inv}{1},
\end{align*}
the second term is
\begin{align}
\eqref{b}
&=
\sum_{l=1}^\infty q^{-\frac{l}{2}} \chi(\varpi)^l
\int_{\calO^\times}
\left\{
1
+ \sum_{m=1}^\infty q^m \int_{\calO^\times} \Phi_0(\mmatrix{1}{0}{\varpi^l v\inv}{1} ([0, \varpi^{-m}u], 0)) du
\right\}  dv\notag \\
&=
\sum_{l=1}^\infty q^{-\frac{l}{2}} \chi(\varpi)^l
\int_{\calO^\times}
\left\{
1
+ \sum_{m=1}^\infty q^m \int_{\calO^\times} \psi(-\varpi^{l-2m}v\inv u^2) \Phi_0(([-\varpi^{l-m}v\inv u, 0], 0)) du
\right\}  dv\notag \\
&=
\sum_{l=1}^\infty q^{-\frac{l}{2}} \chi(\varpi)^l
\int_{\calO^\times}
\left\{
1
+ \sum_{m=1}^l q^m \int_{\calO^\times} \psi(-\varpi^{l-2m}v\inv u^2) du
\right\}  dv\notag \\
&=
\sum_{l=1}^\infty q^{-\frac{l}{2}} \chi(\varpi)^l
\left\{
1-q\inv
+ \sum_{m=1}^l q^m \int_{\calO^\times}\int_{\calO^\times} \psi(-\varpi^{l-2m} u^2 v) dv du
\right\}. \label{c}
\end{align}
By \cite[Lemma 1.9]{szp}, we have
\begin{align*}
\int_{\calO^\times} \int_{\calO^\times} \psi(-\varpi^{l-2m} u^2 v) dv du
=\begin{cases}
(1-q\inv)^2, & l-2m\geq0,\\
-q\inv (1-q\inv), & l-2m=-1,\\
0, & l-2m \leq-2.
\end{cases}
\end{align*}
Thus, \eqref{c} is equal to
\begin{align*}
&\sum_{k=1}^\infty
\sum_{l=2k-1}^{2k} q^{-\frac{l}{2}} \chi(\varpi)^l
\left\{
1-q\inv
+ \sum_{m=1}^l q^m \int_{\calO^\times}\int_{\calO^\times} \psi(-\varpi^{l-2m} u^2 v) dv du
\right\}\\
&=\sum_{k=1}^\infty
q^{-k+\frac{1}{2}} \chi(\varpi)^{2k-1}
\left\{
1-q\inv
+ \sum_{m=1}^{k-1} q^m (1-q\inv)^2
-q^k q\inv(1-q\inv)
\right\}\\
&\quad +
q^{-k} \chi(\varpi)^{2k}
\left\{
1-q\inv
+ \sum_{m=1}^k q^m (1-q\inv)^2
\right\}\\
&=
\sum_{k=1}^\infty
q^{-k+\frac{1}{2}} \chi(\varpi)^{2k-1} \times 0
+ q^{-k} \chi(\varpi)^{2k} \times (1-q\inv)q^k \\
&=
\sum_{k=1}^\infty
(1- q\inv) \chi(\varpi)^{2k}.
\end{align*}
Therefore we obtain
\begin{align*}
\int_F \int_F \Phi_0(\mmatrix{0}{1}{-1}{0}\mmatrix{1}{x}{0}{1} ([0,\mu], 0)) d\mu dx
&=1+\sum_{k=1}^\infty (1- q\inv) \chi(\varpi)^{2k}\\
&=\frac{1-q\inv \chi(\varpi)^2}{1-\chi(\varpi)^2}
\neq0.
\end{align*}

Next, we treat the case of $n=2$.
There are four cases to consider: i)$s_1=s_2=0$, ii)$s_1=s_2>0$, iii)$s_1>s_2=0$, and iv)$s_1>s_2>0$.

i)$s_1=s_2=0$.
Then the unramified principal series representation $I_{B, \overline{\psi}}(\chi_1, \chi_2)$ is irreducible, so the assertion follows from Theorem \ref{ps}.

ii)$s_1=s_2>0$, iii)$s_1>s_2=0$, and iv)$s_1>s_2>0$.
We treat these cases at the same time.
Let $M_2$ and $M_1$ be subgroups of $\Sp_4(F)$ consisting of matrices of the form
\begin{align*}
&\left(\begin{array}{cc}
A&0\\
0&\tp{A}\inv
\end{array}\right),&
A &\in \GL_2(F),
\end{align*}
and
\begin{align*}
&\left(\begin{array}{cccc}
t&&&\\
&a&&b\\
&&t\inv&\\
&c&&d
\end{array}\right),&
t &\in F^\times, \quad
\left(\begin{array}{cc}
a&b\\
c&d
\end{array}\right) \in \Sp_2(F),
\end{align*}
respectively.
Note that $M_2 \cong \GL_2(F)$ and $M_1 \cong \GL_1(F) \times \Sp_2(F)$.
For $j=1,2$, let $P_j$ be a parabolic subgroup of $\Sp_4(F)$ such that $P_j$ contains the Borel subgroup $B$ and that $M_j$ is a Levi subgroup of $P_j$.
Then the unramified principal series representation $I_{B, \overline{\psi}}(\chi_1, \chi_2)$ is a standard module because of isomorphisms
\begin{align*}
I_{B, \overline{\psi}}(\chi_1, \chi_2)
\cong
\begin{cases}
I_{P_2, \overline{\psi}} \left( \Ind_{B_{\GL_2}}^{\GL_2(F)}(\chi_1'\boxtimes\chi_2') \otimes |\det|_F^{s_1} \right) & \text{ii)},\\
I_{P_1, \overline{\psi}} \left( \chi_1'|-|^{s_1}, I_{B, \overline{\psi}}(\chi_2') \right) & \text{iii)}, \\
I_{B, \overline{\psi}} \left(\chi_1'|-|_F^{s_1}, \chi_2'|-|_F^{s_2} \right) & \text{iv)},
\end{cases}
\end{align*}
in the notation of \cite[\S2.5]{gs}, where $B_{\GL_2}$ denotes the Borel subgroup of $\GL_2(F)$ consisting of upper triangular matrices, and $\chi_1'$ and $\chi_2'$ are unitary unramified characters such that $\chi_1'|-|_F^{s_1}=\chi_1$, $\chi_2'|-|_F^{s_2}=\chi_2$.
Put
\begin{align*}
(P,M)
=\begin{cases}
(P_2, M_2)& \text{ii)},\\
(P_1, M_1)& \text{iii)}, \\
(B, T)& \text{iv)}.
\end{cases}
\end{align*}
Let $U$ be the unipotent radical of $P$, and $w \in \Sp_4(F)$ a representative of the longest element in the relative Weyl group $W(M,\Sp_4)=N_{\Sp_4}(M)/M$ for $M\subset\Sp_4(F)$.
Here, $N_{\GL_2}(M)$ denotes the normalizer of $M$ in $\GL_2$.
It is known that the unramified representation $\pi_{\overline{\psi}}(\alpha_1, \alpha_2)$ is the image of an intertwining operator $\calT_w$ from $I_{B, \overline{\psi}}(\chi_1, \chi_2)$ to $I_{B, \overline{\psi}}(\chi_1^\vee, \chi_2^\vee)$ defined by
\begin{align*}
(\calT_w \varphi)(g)
=\int_U \varphi((wu,1)g) du,
\quad g \in \Mp_4(F),
\end{align*}
so it suffices to show that $\calT^J_w \Phi_0 \neq 0$, where $\calT^J_w=\calT_w \otimes 1_{\pi_{\SW, \psi}}$.
Here, $\Phi_0 \in I_{B^J}(\chi_1, \chi_2; \psi)$ is the spherical vector defined before Theorem \ref{csu}.

Put
\begin{align*}
w_1
&=\left(\begin{array}{cccc}
0&1&&\\
-1&0&&\\
&&0&1\\
&&-1&0
\end{array}\right),&
w_2
&=\left(\begin{array}{cccc}
1&&&\\
&0&&1\\
&&1&\\
&-1&&0
\end{array}\right),
\end{align*}
and $w_0=\mmatrix{}{-1_2}{1_2}{}$.
Then $w_1$ and $w_2$ are representatives of simple reflections in the Weyl group $W(T,\Sp_4)=N_{\Sp_4}(T)/T$, and $w_0$ is a representative of the longest element in it.
Moreover, $w_0$ is also a representative of the longest element in $W(M_2,\Sp_4)$ and $W(M_1,\Sp_4)$.
Note that $w_1 \in M_2$ and $w_2 \in M_1$.

Since the image of $\calT_w$ does not depend on the choice of $w$, we may take
\begin{align}\label{eqrev2}
w=w_0
=\left(\begin{array}{cccc}
&&-1&\\
&&&-1\\
1&&&\\
&1&&
\end{array}\right)
=w_2 w_1 w_2 w_1,
\end{align}
in the case iv).
Moreover, because we have $w_1 \in M_2$ and $w_2 \in M_1$, we may take
\begin{align}\label{eqrev3}
w=w_0 w_1\inv
=\left(\begin{array}{cccc}
&&&1\\
&&-1&\\
&-1&&\\
1&&&
\end{array}\right)
=w_2 w_1 w_2,
\end{align}
in the case ii), and
\begin{align}\label{eqrev4}
w=w_2\inv w_0
=\left(\begin{array}{cccc}
&&-1&\\
&-1&&\\
1&&&\\
&&&-1
\end{array}\right)
=w_1 w_2 w_1,
\end{align}
in the case iii).
The equations \eqref{eqrev2}, \eqref{eqrev3}, and \eqref{eqrev4} give the reduced expressions of each of the elements $w$.
Then by the multiplicative property of the intertwining operators in $w$, the intertwining operator $\calT_w$ is equal, up to sign, to a composition
\begin{align*}
\begin{cases}
\calT_2 \circ \calT_1 \circ \calT_2
:I_{B, \overline{\psi}}(\chi_1, \chi_2) \to I_{B, \overline{\psi}}(\chi_1, \chi_2^\vee)
 \to I_{B, \overline{\psi}}(\chi_2^\vee, \chi_1) \to I_{B, \overline{\psi}}(\chi_2^\vee, \chi_1^\vee)
  \quad \left( \cong I_{B, \overline{\psi}}(\chi_1^\vee, \chi_2^\vee) \right) & \text{ii)},\\
\calT_1 \circ \calT_2 \circ \calT_1
:I_{B, \overline{\psi}}(\chi_1, \chi_2) \to I_{B, \overline{\psi}}(\chi_2, \chi_1)
 \to I_{B, \overline{\psi}}(\chi_2, \chi_1^\vee) \to I_{B, \overline{\psi}}(\chi_1^\vee, \chi_2)
  \quad \left( \cong I_{B, \overline{\psi}}(\chi_1^\vee, \chi_2^\vee) \right)& \text{iii)},\\
\calT_2 \circ \calT_1 \circ \calT_2 \circ \calT_1
:I_{B, \overline{\psi}}(\chi_1, \chi_2) \to I_{B, \overline{\psi}}(\chi_2, \chi_1)
\to I_{B, \overline{\psi}}(\chi_2, \chi_1^\vee) \to I_{B, \overline{\psi}}(\chi_1^\vee, \chi_2)
 \to I_{B, \overline{\psi}}(\chi_1^\vee, \chi_2^\vee)& \text{iv)},
\end{cases}
\end{align*}
where $\calT_1$ and $\calT_2$ are given by
\begin{align*}
(\calT_1 \varphi)(g)
&=\int_F \varphi((w_1 \left(\begin{array}{cccc} 1&x&&\\&1&&\\&&1&\\&&-x&1\end{array}\right) ,1)g) dx,& &\\
(\calT_2 \varphi)(g)
&=\int_F \varphi((w_2 \left(\begin{array}{cccc} 1&&&\\&1&&x\\&&1&\\&&&1\end{array}\right) ,1)g) dx,&
g &\in \Mp_4(F).
\end{align*}
Now, $\calT_1$ (resp. $\calT_2$) in the composition operates on a certain unramified principal series representation $I_{B, \overline{\psi}}(\chi', \chi)$ such that $|\chi'\chi\inv(\varpi)|<1$ (resp. $|\chi(\varpi)|<1$).
Thus it suffices to show that $\calT^J_i \Phi_0 \neq 0$ ($i=1,2$) for the spherical vector $\Phi_0 \in I_{B^J}(\chi',\chi;\psi)$, where $\calT^J_i=\calT_i \otimes 1_{\pi_{\SW, \psi}}$.

We shall regard $\Sp^J_2(F)$ as a subgroup of $\Sp^J_4(F)$ by the injection given by
\begin{align*}
\left(\begin{array}{cc}
a&b\\
c&d
\end{array}\right)
&\mapsto
\left(\begin{array}{cccc} 1&&&\\&a&&b\\&&1&\\&c&&d\end{array}\right),&
([\lambda, \mu], \kappa)
&\mapsto
\left(\left[ \left(\begin{array}{c} 0\\ \lambda\end{array}\right), \left(\begin{array}{c} 0\\ \mu\end{array}\right) \right], \kappa\right),
\end{align*}
and consider the restriction $\Phi_0|_{\Sp^J_2(F)}$.
Then the assertion $\calT^J_2 \Phi_0 \neq 0$ follows from the case II) above.

The other assertion can be verified as follows.
By virtue of Theorem \ref{ps}, for any $\Phi \in I_{B^J}(\chi', \chi; \psi)$ we have
\begin{align*}
(\calT^J_1 \Phi) (1)
&=\int_F \Phi(
\left(\begin{array}{cccc}
0&1&&\\
-1&0&&\\
&&0&1\\
&&-1&0
\end{array}\right)
\left(\begin{array}{cccc}
1&x&&\\
&1&&\\
&&1&\\
&&-x&1
\end{array}\right)
)dx.
\end{align*}
Therefore, by the equation \eqref{a1}, we have
\begin{align*}
(\calT^J_1 \Phi_0) (1)
&=\Phi_0(1)
+\sum_{m=1}^\infty \int_{\calO^\times}
\chi'\chi\inv(\varpi)^m q^{-\frac{5}{2}m + \frac{3}{2}m}
\Phi_0(1) q^m dt\\
&=
1 + \sum_{m=1}^\infty (1-q\inv) \chi'\chi\inv(\varpi)^m\\
&=\frac{1-q\inv \chi'\chi\inv(\varpi)}{1-\chi'\chi\inv(\varpi)}
\neq0.\\
\end{align*}

Finally, we treat the case of $n \geq3$.
Since $s_1\geq \cdots \geq s_n\geq0$, there are nonnegative integers $m, n_0, n_1,\ldots,n_m$ such that $n=n_1+\cdots+n_m+n_0$ and
\begin{equation*}
s_1=\cdots=s_{n_1} > s_{n_1+1}=\cdots=s_{n_1+n_2} > \cdots =s_{n_1+\cdots+n_m} > s_{n_1+\cdots+n_m+1} =\cdots= s_{n_1+\cdots+n_m+n_0}=0.
\end{equation*}
If $s_1=\cdots=s_n=0$ or $s_n>0$, we understand $m$ or $n_0$ to be 0, respectively.
When $m=0$, the unramified principal series representation $I_{B, \overline{\psi}}(\chi_1, \ldots, \chi_n)$ is irreducible, so the assertion follows from Theorem \ref{ps}.

Now assume that $m>0$.
As in the case of $n=2$, let $M$ be a subgroup of $\Sp_{2n}(F)$ consisting of matrices of the form
\begin{align*}
&\left(\begin{array}{cccccccc}
A_1&&&&&&&\\
&\ddots&&&&&&\\
&&A_m&&&&&\\
&&&a&&&&b\\
&&&&\tp{A_1}\inv&&&\\
&&&&&\ddots&&\\
&&&&&&\tp{A_m}\inv&\\
&&&c&&&&d
\end{array}\right),&
A_i &\in \GL_{n_i}(F), \quad
\left(\begin{array}{cc}
a&b\\
c&d
\end{array}\right) \in \Sp_{2n_0}(F).
\end{align*}
Note that $M\cong \GL_{n_1}(F)\times\cdots\times\GL_{n_m}(F)\times\Sp_{2n_0}(F)$.
Let $P$ be a proper standard parabolic subgroup of $\Sp_{2n}(F)$ such that $M$ is a Levi subgroup of $P$.
Then the unramified principal series representation $I_{B, \overline{\psi}}(\chi_1, \ldots, \chi_n)$ is a standard module because it is isomorphic to
\begin{align*}
I_{P,\overline{\psi}}
&\left(
\Ind_{B_{\GL_{n_1}}}^{\GL_{n_1}(F)}(\chi_1'\boxtimes\cdots\boxtimes\chi_{n_1}')\otimes|\det|_F^{s_{n_1}},\ldots
\right.\\
&\quad\left. \ldots, \Ind_{B_{\GL_{n_m}}}^{\GL_{n_m}(F)}(\chi_{n_1+\cdots+n_{m-1}+1}'\boxtimes\cdots\boxtimes\chi_{n_1+\cdots+n_{m-1}+n_m}')\otimes|\det|_F^{s_{n_m}},\ 
I_{B,\overline{\psi}}(\chi_{n_1+\cdots+n_m+1}, \ldots, \chi_n) \right),
\end{align*}
in the notation of \cite[\S2.5]{gs}, where $B_{\GL_N}$ denotes the Borel subgroup of $\GL_N(F)$ consisting of upper triangular matrices, and $\chi_i'$ is the unitary unramified character such that $\chi_i'|-|_F^{s_i}=\chi_i$.
Let $U$ be the unipotent radical of $P$, and $w\in\Sp_{2n}(F)$ a representative of the longest element in the relative Weyl group $W(M,\Sp_{2n})=N_{\Sp_{2n}}(M)/M$ for $M$ and $\Sp_{2n}(F)$.
Then the unramified representation $\pi_{\overline{\psi}}(\alpha_1, \ldots, \alpha_n)$ is the image of an intertwining operator $\calT_w$ from $I_{B, \overline{\psi}}(\chi_1, \ldots, \chi_n)$ to $I_{B, \overline{\psi}}(\chi_1^\vee, \ldots, \chi_n^\vee)$ defined by
\begin{equation*}
(\calT_w \varphi)(g)
=\int_U \varphi((wu,1)g) du,
\quad g\in\Mp_{2n}(F),
\end{equation*}
and it suffices to show that  $\calT^J_w \Phi_0 \neq 0$, where $\calT^J_w=\calT_w \otimes 1_{\pi_{\SW, \psi}}$.
Here, $\Phi_0 \in I_{B^J}(\chi_1,\ldots, \chi_n; \psi)$ is the spherical vector defined before Theorem \ref{csu}.
As in the case $n=2$, we may assume that $w$ is a product of a finite number of elements in $\{w_1, \ldots, w_n\}$, where
\begin{align*}
w_i
=\left(\begin{array}{cccccccc}
1_{i-1}&&&&&&&\\
&0&1&&&&&\\
&-1&0&&&&&\\
&&&1_{n-i-1}&&&&\\
&&&&1_{i-1}&&&\\
&&&&&0&1&\\
&&&&&-1&0&\\
&&&&&&&1_{n-i-1}
\end{array}\right), \ (i=1,\ldots,n-1),
\qquad
w_n
=\left(\begin{array}{cccc}
1_{n-1}&&&\\
&0&&1\\
&&1_{n-1}&\\
&-1&&0\\
\end{array}\right).
\end{align*}
For $x\in F$, put
\begin{align*}
u_i(x)
=\left(\begin{array}{cccccccc}
1_{i-1}&&&&&&&\\
&1&x&&&&&\\
&&1&&&&&\\
&&&1_{n-i-1}&&&&\\
&&&&1_{i-1}&&&\\
&&&&&1&&\\
&&&&&x&1&\\
&&&&&&&1_{n-i-1}
\end{array}\right), \ (i=1,\ldots,n-1),
\qquad
u_n(x)
=\left(\begin{array}{cccc}
1_{n-1}&&&\\
&1&&x\\
&&1_{n-1}&\\
&&&1\\
\end{array}\right).
\end{align*}
Then, as in the case of $n=2$, the intertwining operator $\calT_w$ is equal, up to sign, to a composition of a finite number of corresponding operators $\calT_1, \ldots, \calT_n$ defined by
\begin{align*}
(\calT_i \varphi)(g)
=\int_F \varphi((w_i u_i(x) ,1)g) dx, \quad g\in\Mp_{2n}(F).
\end{align*}
Here, each $\calT_i$ in the composition operates on a certain unramified principal series representation $I_{B, \overline{\psi}}(\acute{\chi}_1, \ldots, \acute{\chi}_n)$ such that $|\acute{\chi}_i \acute{\chi}_{i+1}\inv(\varpi)|<1$ if $1\leq i \leq n-1$, and $|\acute{\chi}_n(\varpi)|<1$ if $i=n$.
Thus, it suffices to show that $\calT^J_i \Phi_0\neq0$ ($i=1,\ldots,n$) for the spherical vector $\Phi_0 \in I_{B^J}(\acute{\chi}_1,\ldots, \acute{\chi}_n; \psi)$, where $\calT^J_i=\calT_i\otimes 1_{\pi_{\SW,\psi}}$.

As in the case of $n=2$, we shall regard $\Sp^J_2(F)$ as a subgroup of $\Sp^J_{2n}(F)$ by the injection $\iota_n$ given by
\begin{align*}
\left(\begin{array}{cc}
a&b\\
c&d
\end{array}\right)
&\mapsto
\left(\begin{array}{cccc} 1_{n-1}&&&\\&a&&b\\&&1_{n-1}&\\&c&&d\end{array}\right),&
([\lambda, \mu], \kappa)
&\mapsto
\left(\left[ \left(\begin{array}{c} 0_{n-1}\\ \lambda\end{array}\right), \left(\begin{array}{c} 0_{n-1}\\ \mu\end{array}\right) \right], \kappa\right).
\end{align*}
Moreover, for $i=1,\ldots, n-1$, we shall regard $\Sp^J_4(F)$ as a subgroup of $\Sp^J_{2n}(F)$ by an injection $\iota_i$ given by
\begin{align*}
\left(\begin{array}{cc}
A&B\\
C&D
\end{array}\right)
&\mapsto
\left(\begin{array}{cccccc}
1_{i-1}&&&&&\\
&A&&&B&\\
&&1_{n-i-1}&&&\\
&&&1_{i-1}&&\\
&C&&&D&\\
&&&&&1_{n-i-1}
\end{array}\right),&
([\lambda, \mu], \kappa)
&\mapsto
\left(\left[ \left(\begin{array}{c} 0_{i-1}\\ \lambda \\ 0_{n-i-1}\end{array}\right), \left(\begin{array}{c} 0_{i-1}\\ \mu \\ 0_{n-i-1}\end{array}\right) \right], \kappa\right).
\end{align*}
Let us consider the restriction $\Phi_0|_{\myim \iota_i}$.
Now the assertion $\calT^J_i \Phi_0 \neq 0$ follows from the case $n=2$ or 1.
\end{proof}

\subsection{The real case}
Next we consider representations of the Jacobi group over $\R$.
Let $\psi : \R\to\C^1$ be a nontrivial additive character defined by $\psi(x)=\bfe(x)$ or $\bfe(-x)$.
The Schr\"odinger representation $\pi_{S,\psi}$ of $\calH_n(\R)$ on the Hilbert space $L^2(\R^n)$ is defined by
\begin{equation*}
\left[\pi_{S,\psi}(([\lambda, \mu], \kappa)) f \right](x)
=\psi(\kappa+\tp(2x+\lambda)\mu) f(x+\lambda),
\quad
f \in \calS(\R^n), \ ([\lambda,\mu],\kappa)\in\calH_n(\R),
\end{equation*}
where $\calS(\R^n)$ is the Schwartz space.
The following fact is known as the Stone-von Neumann theorem.
\begin{thm}
\begin{enumerate}[(1)]
\item The Schr\"odinger representation $\pi_{S,\psi}$ is the unique irreducible unitary representation of $\calH_n(\R)$ with central character $\psi$.
\item Any unitary representation of $\calH_n(\R)$ with central character $\psi$ is isomorphic to a direct sum of $\pi_{S,\psi}$.
\end{enumerate}
\end{thm}
The Stone-von Neumann theorem gives us the Weil representation $\omega_\psi$ of the metaplectic group $\Mp_{2n}(\R)$ on $L^2(\R^n)$.
Combining the Schr\"odinger representation $\pi_{S,\psi}$ and the Weil representation $\omega_\psi$, we obtain the Schr\"odinger-Weil representation $\pi_{\SW,\psi}$ of $\Mp^J_{2n}(\R)=\Mp_{2n}(\R)\ltimes \calH_n(\R)$, the metaplectic double covering group of $\Sp^J_{2n}(\R)$.
Note that the Schr\"odinger-Weil representation $\pi_{\SW,\psi}$ is a genuine representation, i.e., it does not factor through $\Mp^J_{2n}(\R)\to\Sp^J_{2n}(\R)$.

If $\pi'$ is a genuine unitary representation of $\Mp_{2n}(\R)$, then a representation $\pi=\pi' \otimes \pi_{\SW,\psi}$ is not genuine and can be regarded as a representation of $\Sp^J_{2n}(\R)$.
Also, we have the following theorem from \cite{sun}.
\begin{thm}[{\cite[Proposition 4.2]{sun}}]\label{ccr}
The correspondence $\pi' \mapsto \pi:=\pi' \otimes \pi_{\SW,\psi}$ gives a bijection between the irreducible genuine unitary representations of $\Mp_{2n}(\R)$ and the irreducible unitary representations of $\Sp^J_{2n}(\R)$ with central character $\psi$.
\end{thm}

Let us write $\fraksp_{2n}$ for the Lie algebra of $\Sp_{2n}$, or equivalently of $\Mp_{2n}$, and put
\begin{align*}
\frakp_\C
&=\Set{
\left(\begin{array}{cc}
A&B\\B&-A
\end{array}\right) \in \fraksp_{2n}(\C)
|
A=\tp{A}, \ B=\tp{B}
},\\
\frakp_\C^\pm
&=\Set{
\left(\begin{array}{cc}
A&\pm iA\\ \pm iA&-A
\end{array}\right) \in \frakp_\C
|
A=\tp{A}
}.
\end{align*}
The standard maximal compact subgroup $K_\infty$ of $\Sp_{2n}(\R)$ is
\begin{align*}
K_\infty
=\set{
\left(\begin{array}{cc} \alpha&\beta\\\-\beta&\alpha\end{array}\right) \in \GL_{2n}(\R)
|
\tp{\alpha}\alpha+\tp{\beta}\beta=1_n, \ \tp{\alpha}\beta=\tp{\beta}\alpha
},
\end{align*}
and the complexification $\frakk_\C$ of the Lie algebra $\frakk$ of $K_\infty$ is
\begin{align*}
\frakk_\C
=\Set{
\left(\begin{array}{cc}
A&B\\-B&A
\end{array}\right) \in \fraksp_{2n}(\C)
|
A=-\tp{A}, \ B=\tp{B}
}.
\end{align*}

For $\bfa=(a_1, \ldots, a_n) \in (\frac{1}{2}+\Z)^n$ with $a_1 \geq \cdots \geq a_n$, let us write $(\rho_\bfa, V_\bfa)$ for the finite dimensional irreducible representation of $\wtil{U(n)}$ of highest weight $\bfa$.
Here, $\wtil{\U(n)}$ denotes the $\det^{\frac{1}{2}}$ cover defined in \cite[p.7]{ada}, and we choose a root system as follows.
Let $\T^n$ be a maximal compact torus of $\U(n)$ consisting of the diagonal matrices of the form $\diag(e^{i\theta_1}, \ldots, e^{i\theta_n})$, ($\theta_j \in \R$), and take an $\R$-basis $\{H_j = i E_{j,j}\}_j$ of the Lie algebra $\frakt=\Lie(\T^n)$.
Let $f_l$ be the linear form on $\frakt_\C$ which sends $H_j$ to $i \delta_{j,l}$.
Then $\{ f_l \}_l$ is an $\R$-basis of $i \frakt^*$, so $i \frakt^*$ can be identified with $\R^n$ by an isomorphism $\sum_l a_l f_l \mapsto (a_1, \ldots, a_n)$, and the root system is given by $\Delta(\fraku(n), \frakt)=\{ \pm(f_j-f_l) \}_{1 \leq j < l \leq n}$.
We shall fix a system of positive roots $\Delta^+(\fraku(n), \frakt)=\{ f_j - f_l\}_{j<l}$.

Following \cite[p.19]{ada}, fix an identification $K_\infty \cong U(n)$ (and $\wtil{K_\infty} \cong \wtil{U(n)}$) by
\begin{align*}
\left(\begin{array}{cc} \alpha&\beta\\\-\beta&\alpha\end{array}\right)
\mapsto
\begin{cases}
\alpha + i\beta, & \psi=\bfe,\\
\alpha - i\beta, & \psi=\overline{\bfe},
\end{cases}
\end{align*}
and we regard $\rho_\bfa$ as a representation of $\wtil{K_\infty}$.
Since $[\frakk_\C, \frakp_\C^\pm] \subset \frakp_\C^\pm$, the differential $d\rho_\bfa : \frakk_\C \to \End_\C(V_\bfa)$ of the representation $\rho_\bfa$ can be extended to a representation of $\frakk_\C \oplus \frakp_\C^\pm$ by setting $d\rho_\bfa(\frakp_\C^\pm)=0$.
Put
\begin{align*}
M^\pm(V_\bfa)
= \frakU(\fraksp_{2n}(\C)) \otimes_{\frakU(\frakk_\C \oplus \frakp_\C^\pm)} V_\bfa,
\end{align*}
where $\frakU$ denotes the functor of universal enveloping algebras.
We fix a positive roots system $\Delta^+(\fraksp_{2n}(\C), \frakt_\C)=\{ f_j - f_l\}_{j<l} \cup \{ f_j + f_l\}_{j\leq l}$, which is compatible with $\Delta^+(\fraku(n), \frakt)$ above.
Put $\varepsilon=\varepsilon_\psi=-i \psi(1/4)$.
Then $\frakp^\varepsilon_\C$ is spanned by the root spaces for $\{ f_j + f_l\}_{j\leq l}$, and the module $M^{-\varepsilon}(V_\bfa)$ has a unique irreducible quotient, for which we shall write $L^{-\varepsilon}(V_\bfa)$.
The $(\fraksp_{2n}(\C), \wtil{K_\infty})$-module $L^{-\varepsilon}(V_\bfa)$ globalizes to an irreducible genuine unitary representation of $\Mp_{2n}(\R)$, for which we shall write $\pi_{\bfa, \psi}$.
The representation $\pi_{\bfa, \psi}$ is a discrete series representation if $a_n > n$.
Also note that the module $L^\pm(V_\bfa)$ is characterized by the property that it contains $V_\bfa$ as a $\frakU(\frakk_\C \oplus \frakp_\C^\pm)$-submodule with multiplicity one.

In addition, for $\bfk=(k_1, \ldots, k_n) \in \Z^n$ with $k_1 \geq \cdots \geq k_n$, we write $(\rho_\bfk, V_\bfk)$ for the finite dimensional irreducible representation of $\U(n)$ of highest weight $\bfk$, and via the identification we shall regard $\rho_\bfk$ as a representation of $K_\infty$.\\

Let us write $\frakh_n$ for the Lie algebra of $\calH_n$, so that the Lie algebra $\fraksp^J_{2n}$ of $\Sp^J_{2n}$ is given by a direct sum $\fraksp_{2n} \oplus \frakh_n$.
Put
\begin{align*}
\frakq_\C
&=\Set{
([\Lambda, M], 0) \in \frakh_n(\C)
|
\Lambda, M \in \C^n
},\\
\frakq_\C^\pm
&=\Set{
([\Lambda, \pm i\Lambda], 0) \in \frakq_\C
|
\Lambda \in \C^n
}.
\end{align*}
Let $\calF=\C[z_1, \ldots, z_n]$ be the Fock model of the Schr\"odinger-Weil representation $\pi_{\SW, \bfe}$.
Recall that $\frakh_n(\C)$, $\frakp_\C$, and $\frakt_\C \subset \frakk_\C$ act on $\calF$ by
\begin{alignat*}{2}
\frakq_\C^- &\ni ([e_j, -ie_j], 0)& &\mapsto -8\pi \frac{d}{dz_j},\\
\frakq_\C^+ &\ni ([e_j, ie_j], 0)& &\mapsto z_j,\\
\Lie(\calZ)_\C &\ni ([0,0], \varkappa)& &\mapsto \bfe(\varkappa),\\
\frakp_\C^- &\ni \left(\begin{array}{cc} F_{j,k}& -iF_{j,k}\\ -iF_{j,k}& -F_{j,k} \end{array} \right)& &\mapsto 16\pi\frac{d^2}{dz_j dz_k}\\
\frakp_\C^+ &\ni \left(\begin{array}{cc} F_{j,k}& i F_{j,k}\\ i F_{j,k}& -F_{j,k} \end{array} \right)& &\mapsto -\frac{1}{4\pi}z_j z_k,\\
\frakt_\C &\ni \left(\begin{array}{cc} 0& -i E_{j,j}\\ i E_{j,j}&0 \end{array} \right)& &\mapsto \frac{1}{2} \left( z_j \frac{d}{dz_j} + \frac{d}{dz_j} z_j \right),
\end{alignat*}
where $F_{j,k}=E_{j,k}+E_{k,j}$.

\begin{lem}\label{lw}
Let $\pi'$ be an irreducible genuine unitary representation of $\Mp_{2n}(\R)$, $\bfk=(k_1, \ldots, k_n)$ an element of $\Z^n$ with $k_1 \geq \cdots \geq k_n$, and put $\bfa=(k_1-\frac{\varepsilon_\psi}{2}, \ldots, k_n-\frac{\varepsilon_\psi}{2})$.
\begin{enumerate}[(1)]
\item Assume that $\psi=\bfe$.
If the unitary representation $\pi' \otimes \pi_{\SW, \bfe}$ has a nonzero subspace $W\neq0$ such that $\frakq_\C^- \cdot W = \frakp_\C^- \cdot W =0$ and $W \cong V_\bfk$ as a representation of $K_\infty$, then $\pi' \cong \pi_{\bfa, \psi}$.
\item Assume that $\psi=\overline{\bfe}$.
If the unitary representation $\pi' \otimes \pi_{\SW, \bfe}$ has a nonzero subspace $W\neq0$ such that $\frakq_\C^- \cdot W =0$,  $(X^++Y^+_1Y^+_2+Y^+_2Y^+_1) \cdot W =0$ for any $X^+=\mmatrix{A}{iA}{iA}{-A}\in\frakp_\C^+$, $Y^+_j = ([\Lambda_j, i\Lambda_j], 0) \in \frakq_\C^+$ ($j=1,2$) with $A=8\pi(\Lambda_1\tp{\Lambda_2}+\Lambda_2\tp{\Lambda_1})$, and $W \cong V_\bfk$ as a representation of $K_\infty$, then $\pi' \cong \pi_{\bfa, \psi}$.
\end{enumerate}
\end{lem}
\begin{proof}
In both cases, we may assume $W \subset \pi' \otimes \calF$, and since $\frakq_\C^- \cdot W =0$ we have $W = W' \otimes 1$, where $W'$ is a subspace of $\pi'$.

(1)
Since $\frakp_\C^- \cdot W =0$, we have $\frakp_\C^- \cdot W'=0$.
The assumption that $W \cong V_\bfk$ and the fact that the weight of the action of $\wtil{K_\infty}$ on $1 \in \calF$ is $(\frac{1}{2}, \ldots, \frac{1}{2})$ imply that $W' \cong V_\bfa$.
Therefore, we obtain that $\pi' \cong \pi_{\bfa, \psi}$.

(2)
Let $X^+\in\frakp_\C^+$ and $Y^+_j  \in \frakq_\C^+$ ($j=1,2$) be as in the assertion.
Since $(X^++Y^+_1Y^+_2+Y^+_2Y^+_1) \cdot 1 =0$ in the Fock model $\calF$, we have $(X^++Y^+_1Y^+_2+Y^+_2Y^+_1) \cdot W' =0$.
Because $\frakq_\C$ acts on $W'$ by 0, this means that $X^+ \cdot W' =0$.
The assumption that $W \cong V_\bfk$ and the fact that the weight of the action of $\wtil{K_\infty}$ on $1 \in \calF$ is $(-\frac{1}{2}, \ldots, -\frac{1}{2})$ imply that $W' \cong V_\bfa$.
Therefore, we obtain that $\pi' \cong \pi_{\bfa, \psi}$.
\end{proof}

\subsection{The global case}
In this subsection, we shall consider representations of the adelic Jacobi group.
Let $F$ be a number field, $\A_F$ the ring of adeles of $F$, and $\psi$ an additive character of $F\backslash \A_F$.
The Schr\"odinger representation $\pi_{S,\psi}$ of $\calH_n(\A_F)$ on the Hilbert space $L^2(\A_F^n)$ is defined by
\begin{equation*}
\left[\pi_{S,\psi}(([\lambda, \mu], \kappa)) f \right](x)
=\psi(\kappa+\tp(2x+\lambda)\mu) f(x+\lambda),
\quad
f \in \calS(\A_F^n), \ ([\lambda,\mu],\kappa)\in\calH_n(\A_F),
\end{equation*}
where $\calS(\A_F^n)$ denotes the Schwartz space.
The following fact is known as the Stone-von Neumann theorem.
\begin{thm}
\begin{enumerate}[(1)]
\item The Schr\"odinger representation $\pi_{S,\psi}$ is the unique irreducible smooth unitary representation of $\calH_n(\A_F)$ with central character $\psi$.
\item Any smooth unitary representation of $\calH_n(\A_F)$ with central character $\psi$ is isomorphic to a direct sum of $\pi_{S,\psi}$.
\end{enumerate}
\end{thm}
The Stone-von Neumann theorem gives us the Weil representation $\omega_\psi$ of the metaplectic group $\Mp_{2n}(\A_F)$ on $L^2(\A_F^n)$.
Combining the Schr\"odinger representation $\pi_{S,\psi}$ and the Weil representation $\omega_\psi$, we obtain the Schr\"odinger-Weil representation $\pi_{\SW,\psi}$ of $\Mp^J_{2n}(\A_F)=\Mp_{2n}(\A_F)\ltimes \calH_n(\A_F)$, the metaplectic double covering group of $\Sp^J_{2n}(\A_F)$.
Note that the Schr\"odinger-Weil representation $\pi_{\SW,\psi}$ is a genuine representation, i.e., it does not factor through $\Mp^J_{2n}(\A_F)\to\Sp^J_{2n}(\A_F)$.

We can realize the Schr\"odinger-Weil representation on a space of functions on the group $\Mp^J_{2n}(\A_F)$, by using theta functions.
For any $f \in \calS(\A_F^n)$, the theta function $\Theta_f$ is defined by
\begin{align*}
\Theta_f(g)
=\sum_{\xi \in F} [\pi_{\SW, \psi} (g) f] (\xi),
\quad
g \in \Mp^J_{2n}(\A_F).
\end{align*}
It converges absolutely and satisfies that $\Theta_{\pi_{\SW, \psi}(g) f} (x)= \Theta_f (xg)$.
The assignment can be extended to $L^2(\A_F)$.\\

Let $L^2(\Sp^J_{2n}(F) \backslash \Sp^J_{2n}(\A_F))_\psi$ be a Hilbert space consisting of measurable functions $\Phi : \Sp^J_{2n}(\A_F) \to \C$ which satisfy
\begin{align*}
\Phi(\gamma g z)
&= \psi(\kappa) \Phi(g),&
\gamma &\in \Sp^J_{2n}(F), \quad  g \in \Sp^J_{2n}(\A_F), \quad  z=([0,0], \kappa) \in \calZ(\A_F),
\end{align*}
and
\begin{align*}
\int_{\Sp^J_{2n}(F) \calZ(\A_F) \backslash \Sp^J_{2n}(\A_F)} | \Phi(g) |^2 dg
< \infty,
\end{align*}
and $L^2_{\mathrm{cusp}}(\Sp^J_{2n}(F) \backslash \Sp^J_{2n}(\A_F))_\psi$ its subspace consisting of cuspidal functions.
Here, a function $\Phi$ on $\Sp^J_{2n}(\A_F)$ is said to be cuspidal if
\begin{align*}
&\int_{N^J(F) \backslash N^J(\A_F)} \Phi(ng) dn =0,&
g &\in \Sp^J_{2n}(\A_F),
\end{align*}
for the unipotent radical $N$ of any proper standard parabolic subgroup of $\Sp_{2n}$, where $N^J$ denotes the subgroup of $\Sp^J_{2n}$ generated by $N$ and elements of the form $([0,\mu],0)$.
Then we have the following fact.
(The proof is the same as \cite[Theorem 7.3.3]{bs}.)
\begin{thm}\label{cca}
There is a natural isomorphism of Hilbert spaces
\begin{align*}
L^2(\Mp_{2n}) \otimes L^2(\A_F^n)
&\overset{\simeq}{\lra} L^2(\Sp^J_{2n}(F) \backslash \Sp^J_{2n}(\A_F))_\psi, \\
\varphi \otimes f
&\mapsto \left[ gh \mapsto \varphi((g, 1)) \Theta_f(gh) \right],
\quad g \in \Sp_{2n}(\A_F), \ h \in \calH_n(A_F).
\end{align*}
This is also an isomorphism of unitary representations.
Moreover, the restriction to cuspidal functions gives another isomorphism
\begin{align*}
L^2_{\mathrm{cusp}}(\Mp_{2n}) \otimes L^2(\A_F^n)
&\overset{\simeq}{\lra} L^2_{\mathrm{cusp}}(\Sp^J_{2n}(F) \backslash \Sp^J_{2n}(\A_F))_\psi.
\end{align*}
Here, $L^2_{\mathrm{cusp}}(\Mp_{2n})$ denotes the cuspidal part of $L^2(\Mp_{2n})$.
\end{thm}
As stated in \cite[p. 182]{bs}, the preceding theorem implies the following corollary.
\begin{cor}\label{ccg}
The map $\pi' \mapsto \pi = \pi'\otimes\pi_{\SW, \psi}$ gives a bijective correspondence between genuine automorphic (resp. cuspidal) representations of $\Mp_{2n}(\A_F)$ and automorphic (resp. cuspidal) representation of $\Sp^J_{2n}(\A_F)$ with central character $\psi$, and if
\begin{align*}
\pi' = \bigotimes_v \pi'_v,
\end{align*}
then
\begin{align*}
\pi= \bigotimes_v (\pi'_v \otimes \pi_{\SW, \psi_v}).
\end{align*}
\end{cor}

\section{Adelic lifts of Siegel modular forms and Jacobi forms}\label{Adelic lift}
In this section, we review the adelic lifts of Siegel cusp forms and holomorphic or skew-holomorphic Jacobi cusp forms, and then prove the main theorems.
\subsection{The adelic lift of $f \in S_{j+3, 2k-6}(\Sp_4(\Z))$}\label{adlifato}
Before considering the adelic lift, we fix an identification $\PGSp_4 \cong \SO_5$ given in the proof of the following lemma.
\begin{lem}
We have an isomorphism $\PGSp_4 \cong \SO_5$ over any field $F$ of characteristic not 2.
\end{lem}
\begin{proof}
Define a 5-dimensional symmetric bilinear space $(X, b(-,-))$ by
\begin{align*}
&X=\set{x \in \mathrm{M}_{4\times4}(F) | \tp{x} J_2 - J_2 x=0, \ \tr(x)=0},&&
&b(x,y)=\tr(xy).&&
\end{align*}
Then, $\GSp_4(F)$ acts on the space by
\begin{align*}
\GSp_4(F) \times X \ni (g, x) \mapsto gxg\inv \in X.
\end{align*}
This gives an isomorphism $\PGSp_4(F) \cong \SO(X)$.
Finally, we identify $\SO(X)$ with $\SO_5(F)$ by using a basis $\{x_2, x_1, x_0, x_{-1}, x_{-2}\}$ of $X$, where
\begin{align*}
x_2&=
\left(\begin{array}{cccc}
&&&1\\
&&-1&\\
&-1&&\\
1&&&
\end{array}\right),&
x_1&=
\left(\begin{array}{cccc}
&1&&\\
1&&&\\
&&&1\\
&&1&
\end{array}\right),&
x_0&=
\left(\begin{array}{cccc}
1&&&\\
&-1&&\\
&&1&\\
&&&-1
\end{array}\right),&\\
x_{-1}&=
\left(\begin{array}{cccc}
&1&&\\
-1&&&\\
&&&-1\\
&&1&
\end{array}\right),&
x_{-2}&=
\left(\begin{array}{cccc}
&&&1\\
&&-1&\\
&1&&\\
-1&&&
\end{array}\right).&
\end{align*}
We obtain an isomorphism from $\PGSp_4$ to $\SO_5$ as algebraic groups over $F$.
\end{proof}

Let $j\geq0$ be an even integer and $k\geq3$ a positive integer.
Let $f \in S_{j+3, 2k-6}(\Sp_4(\Z))$ be a Hecke eigenform.
As in \cite[\S 4.5 and \S 6.3.4]{cl}, we construct an irreducible cuspidal automorphic representation $\pi_f$ of $\SO_5(\A)$, and let $\phi_f$ be its $A$-parameter.
Note that by Theorem \ref{amf}, there exists one and only one elliptic $A$-parameter of $\pi_f$.

One can determine the form of the $A$-parameter $\phi_f$.
For any $a \in \frac{1}{2} \Z$, we shall write $\calD_a$ for the representation of $W_\R$ induced from the character $z=re^{i\theta} \mapsto e^{2ia\theta}$ of $W_\C=\C^\times$.
Then we have the following lemma.
\begin{lem}\label{api}
There is a unique irreducible everywhere unramified cuspidal symplectic automorphic representation $\tau_f$ of $\GL_4(\A)$ such that
\begin{align}\label{eqrev1}
\begin{aligned}
L(s,f,\mathrm{spin})
&=L(s-j-k+\tfrac{3}{2}, \tau_f),\\
\tau_{f,\infty}
&=\calD_{k+j-\frac{3}{2}} \oplus \calD_{k-\frac{5}{2}}.
\end{aligned}
\end{align}
\end{lem}
\begin{proof}
Our assertion is almost the same as that of \cite[Proposition 9.1.4]{cl} but slightly different, since we treat the case of $k\geq 3$ and assume that $j$ is even.
Note that our $j$ and $k$ are different from $j$ and $k$ in \cite[Proposition 9.1.4]{cl}.

Define a maximal compact torus $S$ of $\SO_5(\R)$ by
\begin{align*}
S
&=\Set{
s(\theta_1, \theta_2)
=\left(\begin{array}{ccccc}
\cos\theta_1& \sin\theta_1&&&\\
-\sin\theta_1&\cos\theta_1&&&\\
&&1&&\\
&&&\cos\theta_2&\sin\theta_2\\
&&&-\sin\theta_2&\cos\theta_2
\end{array}
\right)
\in \SO_5(\R)
|
\theta_1, \theta_2 \in \R
},
\end{align*}
and fix a basis $\{b_1, b_2\}$ of the group $X^*(S)$ of characters given by
\begin{align*}
b_l(s(\theta_1, \theta_2))
&= e^{i\theta_l}.
\end{align*}

The irreducible representation $\pi_{f, \infty}$ of $\SO_5(\R)$ is a discrete series representation with the Blattner parameter $(k-3, k+j)$, and its $L$-parameter is $\calD_{k+j-\frac{3}{2}} \oplus \calD_{k-\frac{5}{2}}$. 
By \cite{mr1, mr2}, if $k \geq 4$ the only local $A$-parameter such that its local $A$-packet contains $\pi_{f, \infty}$ is
\begin{align*}
\calD_{k+j-\frac{3}{2}} \boxtimes S_1 \oplus \calD_{k-\frac{5}{2}}\boxtimes S_1.
\end{align*}
On the other hand, if $k=3$, the local $A$-parameters of which the local $A$-packets contain $\pi_{f, \infty}$ are
\begin{align}\label{rp2}
\calD_{j+\frac{3}{2}} \boxtimes S_1 &\oplus \calD_{\frac{1}{2}}\boxtimes S_1,&
\calD_{j+\frac{3}{2}} \boxtimes S_1 &\oplus \chi_a \boxtimes S_2,
\end{align}
where $\chi_a$ is a quadratic character of $W_\R$.
In the packet $\Pi_{\calD_{j+\frac{3}{2}} \boxtimes S_1 \oplus \chi_a \boxtimes S_2}(\SO_5)$, the representation $\pi_{f, \infty}$ corresponds to the character
\begin{align}\label{ch1}
S_{\calD_{j+\frac{3}{2}} \boxtimes S_1 \oplus \chi_a \boxtimes S_2} \cong (\Z/2\Z)^2
\ni (c,d) \mapsto (-1)^{c+d}.
\end{align}

The $A$-parameter $\phi_f$ for $\SO_5$ is one of the following forms:
\begin{enumerate}[(1)]
\item $\phi_f=\chi \boxtimes S_4$, where $\chi$ is a quadratic character of $\Q^\times \backslash \A^\times$;
\item $\phi_f=\chi \boxtimes S_2 \oplus \chi' \boxtimes S_2$, where $\chi$ and $\chi'$ are distinct quadratic characters of $\Q^\times \backslash \A^\times$;
\item $\phi_f=\sigma \boxtimes S_2$, where $\sigma$ is an irreducible cuspidal orthogonal automorphic representation of $\GL_2(\A)$;
\item $\phi_f=\chi \boxtimes S_2 \oplus \sigma \boxtimes S_1$, where $\chi$ is a quadratic character of $\Q^\times \backslash \A^\times$, and $\sigma$ is an irreducible cuspidal symplectic automorphic representation of $\GL_2(\A)$;
\item $\phi_f=\sigma \boxtimes S_1 \oplus \sigma' \boxtimes S_1$, where $\sigma$ and $\sigma'$ are distinct irreducible cuspidal symplectic automorphic representations of $\GL_2(\A)$;
\item $\phi_f=\tau \boxtimes S_1$, where $\tau$ is an irreducible cuspidal symplectic automorphic representation of $\GL_4(\A)$.
\end{enumerate}

Since $\pi_f$ is unramified everywhere, $\chi$, $\chi'$, $\sigma$, $\sigma'$, $\tau$, and $\tau'$ are also unramified everywhere.
Moreover, $\chi$ must be trivial since $\A^\times =\Q^\times \R^\times_{>0} \hat{\Z}^\times$.
When $k>3$, by \cite[Proposition 9.1.4]{cl}, we have $\phi_f=\tau \boxtimes S_1$ as in the case (6).
Let us consider the case that $k$ is 3.
Since the localization $\phi_{f, \infty}$ at the real place must be one of the forms \eqref{rp2}, the cases (1), (2), and (3) cannot occur.
By the proof of \cite[Proposition 9.1.4]{cl}, neither do the case (5).

Now consider the case (4).
The localization of $\phi_f$ at the real place must be $\phi_{f, \infty}=1\boxtimes S_2 \oplus \calD_{j+\frac{3}{2}} \boxtimes S_1$.
Then Arthur's character $\epsilon_{\phi_f}$ of $S_{\phi_f} \cong (\Z/2\Z)^2$ is given by
\begin{align*}
\epsilon_{\phi_f}(c,d)
=\epsilon(\tfrac{1}{2}, \sigma \times \chi)^{c+d}
=\epsilon(\tfrac{1}{2}, \calD_{j+\frac{3}{2}}, \bfe)^{c+d}
=i^{(2j+3+1)(c+d)}.
\end{align*}
Since $j$ is even, this implies that $\epsilon_{\phi_f}$ is trivial.
However, $\pi_f$ is unramified at every finite place and $\pi_{f, \infty}$ corresponds to the character given by \eqref{ch1}.
The case (4) is impossible since this contradicts the multiplicity formula for $\SO_5$.\\

In conclusion, we have $\phi_f = \tau_f \boxtimes S_1$, where $\tau_f$ is an irreducible cuspidal symplectic automorphic representation of $\GL_4(\A)$ such that
\begin{equation*}
\tau_{f,\infty}
=\calD_{k+j-\frac{3}{2}} \oplus \calD_{k-\frac{5}{2}}.
\end{equation*}
By using a similar argument to the proof of \cite[Theorem 3]{as}, one can check that
\begin{equation*}
L(s,f,\mathrm{spin})
=L(s-j-k+\tfrac{3}{2}, \tau_f).
\end{equation*}

Conversely, an irreducible everywhere unramified cuspidal symplectic automorphic representation $\tau_f$ of $\GL_4(\A)$ satisfying the equation \eqref{eqrev1} gives the $A$-parameter of $\pi_f$.
The uniqueness follows from that of the $A$-parameter.
\end{proof}

Since $\tau_f$ is uniquely determined by $f$, we obtain an assignment $f \mapsto \tau_f$.
We have the following lemma.
\begin{lem}\label{iap}
Let $\tau$ be an irreducible cuspidal symplectic automorphic representation of $\GL_4(\A)$ that is unramified everywhere and satisfies $\tau_\infty=\calD_{k+j-\frac{3}{2}} \oplus \calD_{k-\frac{5}{2}}$.
Then there exists a Hecke eigenform $f \in S_{j+3, 2k-6}(\Sp_4(\Z))$ such that $\tau_f=\tau$, and it is unique up to a scalar multiple.
\end{lem}
\begin{proof}
The proof is the same as \cite[Proposition 9.1.4 (i) and (iii)]{cl}.
\end{proof}

\subsection{The adelic lifts of Jacobi forms of general degree}\label{aljg}
In this subsection, we consider the adelic lifts of holomorphic or skew-holomorphic Jacobi cusp forms of general degree.

Let $F \in J_{\rho, 1}^{\star, \mathrm{cusp}}$ be a Hecke eigenform, where $\star \in \{\mathrm{hol}, \mathrm{skew} \}$ and $(\rho, V)$ is an irreducible finite dimensional polynomial representation of $\GL_n(\C)$.
The strong approximation theorem for $\Sp_{2n}$ induces that for $\Sp^J_{2n}$:
\begin{align*}
\Sp^J_{2n}(\A)=\Sp^J_{2n}(\Q) \Sp^J_{2n}(\R) \Sp^J_{2n}(\hat{\Z}).
\end{align*}
Now the adelic lift $\Phi_F$ of $F$ is defined by
\begin{align*}
&\Phi_F(g)
=[F|_{\rho, 1}^\star  g_\infty] (i,0),&
&g=\gamma g_\infty \kappa \in \Sp^J_{2n}(\A)=\Sp^J_{2n}(\Q) \Sp^J_{2n}(\R) \Sp^J_{2n}(\hat{\Z}).&
\end{align*}
\begin{lem}\label{adlif}
The function $\Phi_F : \Sp^J_{2n}(\A) \to V$ satisfies the following:
\begin{enumerate}[(1)]
\item $\Phi_F$ is left $\Sp^J_{2n}(\Q)$-invariant;
\item $\Phi_F$ is right $\Sp^J_{2n}(\hat{\Z})$-invariant;
\item
$\Phi_F(grz)=
\begin{cases}
\bfe(\kappa) \overline{\rho(\alpha+i\beta)}\inv \Phi_F(g),&  \star=\mathrm{hol},\\
\bfe(\kappa) \det(\alpha+i\beta) \rho(\alpha+i\beta)\inv \Phi_F(g),&  \star=\mathrm{skew},
\end{cases}
$\\
for any $g \in \Sp^J_{2n}(\R)$, $r=\mmatrix{\alpha}{\beta}{-\beta}{\alpha} \in K_\infty$, $z=([0,0],\kappa) \in \calZ(\R)$;
\item For any $g\in \Sp^J_{2n}(\A)$, the mapping $\Sp^J_{2n}(\R) \ni x\mapsto\Phi_F(xg) \in V$ is $C^\infty$;
\item $
\begin{cases}
X^-\cdot \Phi_F = Y^-\cdot\Phi_F=0,&  \star=\mathrm{hol},\\
(X^++Y^+_1Y^+_2+Y^+_2Y^+_1)\cdot \Phi_F=Y^-\cdot\Phi_F=0,& \star=\mathrm{skew},
\end{cases}
$\\
for any $X^- \in \frakp_\C^-$, $Y^- \in \frakq_\C^-$, and $X^+=\mmatrix{A}{iA}{iA}{-A}\in\frakp_\C^+$, $Y^+_j = ([\Lambda_j, i\Lambda_j], 0) \in \frakq_\C^+$ ($j=1,2$) such that $A=8\pi(\Lambda_1\tp{\Lambda_2}+\Lambda_2\tp{\Lambda_1})$;
\item $\Phi_F$ is cuspidal.
\end{enumerate}
\end{lem}
\begin{proof}
(1)-(4) follow immediately from the definitions, and the proof of assertion (6) is similar to that of \cite[Lemma 5]{as}.

Let us prove (5).
If $\star=\mathrm{hol}$, then $F(Z,w)$ is holomorphic in $(Z,w) \in \frakH_n \times \C^n$, and if $\star=\mathrm{skew}$, then $F(Z,w)$ is holomorphic in $w \in \C^n$ for any $Z \in \frakH_n$.
Thus by a similar argument to \cite[Lemma 7]{as}, one can see
\begin{align*}
\begin{cases}
X^-\cdot \Phi_F = Y^-\cdot\Phi_F=0,&  \star=\mathrm{hol},\\
Y^-\cdot\Phi_F=0,& \star=\mathrm{skew},
\end{cases}
\end{align*}
for any $X^- \in \frakp_\C^-$, $Y^- \in \frakq_\C^-$.

Now we assume $\star=\mathrm{skew}$ and show that $(X^++Y^+_1Y^+_2+Y^+_2Y^+_1)\cdot \Phi_F=0$ for $X^+=\mmatrix{A}{iA}{iA}{-A}\in\frakp_\C^+, Y^+_j = ([\Lambda_j, i\Lambda_j], 0) \in \frakq_\C^+$ ($j=1,2$) such that $A=8\pi(\Lambda_1\tp{\Lambda_2}+\Lambda_2\tp{\Lambda_1})$.
We may assume that $\Lambda_j \in \R^n$, (j=1,2) and $A\in\Sym_n(\R)$.
For any $g \in \Sp^J_{2n}(\R)$, we have
\begin{align*}
[X^+\cdot\Phi_F](g)
=\left. \frac{d}{dt}\right|_{t=0} \Phi_F \left( ge^{t\xi_1} \right)
+i \left. \frac{d}{dt}\right|_{t=0} \Phi_F \left( ge^{t\xi_2}\right)
+i \left. \frac{d}{dt}\right|_{t=0} \Phi_F \left( ge^{t\xi_3} \right),
\end{align*}
where $\xi_1 = \mmatrix{A}{0}{0}{-A}$, $\xi_2 = \mmatrix{0}{A}{0}{0}$, and $\xi_3 = \mmatrix{0}{0}{A}{0}$.
Put $H=F|_\rho^\mathrm{skew} g$.
Then the first term is equal to
\begin{align*}
\left. \frac{d}{dt}\right|_{t=0} \rho(e^{tA}) H(ie^{2tA},0)
=d\rho(A) H(i,0) + \tr\left(\frac{\partial H}{\partial Y}(i,0) \cdot 2A \right),
\end{align*}
and the second is
\begin{align*}
i\left. \frac{d}{dt}\right|_{t=0} H(i + tA,0)
=i \tr\left(\frac{\partial H}{\partial X} (i,0) \cdot A\right),
\end{align*}
where we write
\begin{align*}
&\frac{\partial}{\partial X}=\left( 2^{\delta_{i,j}-1} \frac{\partial}{\partial x_{i,j}} \right)_{i,j},
\quad
\frac{\partial}{\partial Y}=\left( 2^{\delta_{i,j}-1} \frac{\partial}{\partial y_{i,j}} \right)_{i,j},&
&Z=X+iY.&
\end{align*}
Here note that $X$ and $Y$ are symmetric.
The third term is equal to
\begin{align}\label{skhi1}
i \left. \frac{d}{dt}\right|_{t=0} \frac{|\det(1+itA)|}{\det(1+itA)} \rho(1-itA)\inv H(i (itA+1)\inv, 0).
\end{align}
To calculate this, let $\alpha_1, \ldots,\alpha_n \in\C$ be the eigenvalues of $A$.
Then we have
\begin{align*}
\left. \frac{d}{dt}\right|_{t=0}\frac{|\det(1+itA)|}{\det(1+itA)}
&=\left. \frac{d}{dt}\right|_{t=0} \prod_{j=1}^n\frac{|1+it\alpha_j|}{1+it\alpha_j}\\
&=\sum_{j=1}^n \left. \frac{d}{dt}\right|_{t=0} (1+\alpha_j^2t^2)^\frac{1}{2} (1+it\alpha_j)\inv
=-i\tr(A).
\end{align*}
Therefore \eqref{skhi1} is equal to
\begin{align*}
\tr(A)H(i,0) -d\rho(A) H(i,0) + i \tr\left(\frac{\partial H}{\partial X}(i,0) \cdot A\right).
\end{align*}
Thus, we obtain that
\begin{align*}
[X^+\cdot\Phi_F](g)
=\tr(A)H(i,0) +4i \tr\left(\frac{\partial H}{\partial Z} (i,0) A\right),
\end{align*}
where $\frac{\partial}{\partial Z}=\frac{1}{2}(\frac{\partial}{\partial X} - i\frac{\partial}{\partial Y})$.

Next, we calculate $(Y^+_1Y^+_2 + Y^+_2Y^+_1)\cdot \Phi_F$.
By definition, we have
\begin{align*}
[Y^+_1Y^+_2\Phi_F](g)
&=\left. \frac{d}{dt}\right|_{t=0} [Y^+_2\Phi_F](ge^{([t\Lambda_1,0],0)})
 +i\left. \frac{d}{dt}\right|_{t=0}[Y^+_2\Phi_F](ge^{([0,t\Lambda_1],0)})\\
&=\left. \frac{d^2}{dtds}\right|_{t=s=0} \Phi_F(ge^{([t\Lambda_1,0],0)}e^{([s\Lambda_2,0],0)})
 +i\left. \frac{d^2}{dtds}\right|_{t=s=0} \Phi_F(ge^{([t\Lambda_1,0],0)}e^{([0,s\Lambda_2],0)})\\
&\quad+i\left. \frac{d^2}{dtds}\right|_{t=s=0} \Phi_F(ge^{([0,t\Lambda_1],0)}e^{([s\Lambda_2,0],0)})
 -\left. \frac{d^2}{dtds}\right|_{t=s=0} \Phi_F(ge^{([0,t\Lambda_1],0)}e^{([0,s\Lambda_2],0)}).
\end{align*}
The first term is
\begin{align*}
\left. \frac{d^2}{dtds}\right|_{t=s=0} \Phi_F(ge^{([t\Lambda_1,0],0)}e^{([s\Lambda_2,0],0)})
&=\left. \frac{d^2}{dtds}\right|_{t=s=0} \bfe(it^2\tp{\Lambda_1}\Lambda_1+2ist\tp{\Lambda_1}\Lambda_2)\bfe(is^2\tp{\Lambda_2}\Lambda_2)H(i,is\Lambda_2+it\Lambda_1)\\
&=-4\pi\tp{\Lambda_1}\Lambda_2H(i,0) + \tr\left(\frac{\partial^2 H}{\partial v^2}(i,0) \Lambda_2\tp{\Lambda_1}\right),
\end{align*}
where we write
\begin{align*}
&\frac{\partial^2}{\partial v^2}
=(\frac{\partial^2}{\partial v_i \partial v_j})_{i,j},&
&w= u+ i v \in \C^n.&
\end{align*}
Similarly, we can calculate the other terms, and thus we obtain
\begin{align*}
[Y^+_1Y^+_2\Phi_F](g)
=-8\pi \tp{\Lambda_1}\Lambda_2 H(i,0) -4 \tr\left(\frac{\partial^2 H}{\partial w^2}(i,0) \Lambda_2\tp{\Lambda_1}\right),
\end{align*}
where
\begin{align*}
\frac{\partial^2}{\partial w^2}
=(\frac{\partial^2}{\partial w_i \partial w_j})_{i,j}.
\end{align*}
Therefore 
\begin{align*}
[(X^++Y^+_1Y^+_2+Y^+_2Y^+_1)\Phi_F](g)
&=\left(\tr(A)-8\pi (\tp{\Lambda_1}\Lambda_2+\tp{\Lambda_2}\Lambda_1) \right) H(i,0)\\
&\quad+4i \tr\left(\frac{\partial H}{\partial Z} (i,0) A\right) -4 \tr\left(\frac{\partial^2 H}{\partial w^2}H(i,0) (\Lambda_2\tp{\Lambda_1}+\Lambda_1\tp{\Lambda_2})\right)\\
&=\frac{1}{2\pi} \tr\left( (8\pi i\frac{\partial}{\partial Z}-\frac{\partial^2}{\partial w^2})H (i,0) A\right).
\end{align*}

Now it is enough to show that $(8\pi i\frac{\partial}{\partial Z}-\frac{\partial^2}{\partial w^2})H (i,0)$ is zero for any $g \in \Sp^J_{2n}(\R)$.
We shall show this in Lemma \ref{cut} below.
\end{proof}

We fix a nondegenerate bilinear form $\an{-,-}$ on $V \times V$ such that
\begin{align*}
\an{\rho(x)v_1, \overline{\rho}(x)v_2}
&=\an{v_1,v_2},&
v_1, v_2 &\in V, \quad x \in \U(n),
\end{align*}
where $\overline{\rho}(x)=\rho(\overline{x})$, and set $\varphi_{F,v}(g)=\an{v, \Phi_F(g)}$, for each $v \in V$.
By Lemma \ref{adlif} (1), (2), and (6), the functions $\varphi_{F, v} : \Sp^J_{2n}(\A) \to \C$, ($v \in V$) are left $\Sp^J_{2n}(\Q)$-invariant, right $\Sp^J_{2n}(\hat{\Z})$-invariant, and cuspidal.
Let $\pi_F$ be a cuspidal automorphic representation of $\Sp^J_{2n}(\A)$ generated by $\varphi_{F, v}$.
By Lemma \ref{adlif} (3), the central character of $\pi_F$ is the nontrivial additive character $\psi$ of $\Q\backslash\A$ such that $\psi_\infty=\bfe$.
Thus,
\begin{align*}
\pi_F \subset L^2_{\mathrm{cusp}}(\Sp^J_{2n}(\Q) \backslash \Sp^J_{2n}(\A))_{\psi}.
\end{align*}
Moreover, Lemma \ref{adlif} (3) shows that
\begin{align*}
r \cdot \varphi_{F, v}
=\begin{cases}
\varphi_{F, \rho(r)v}, & \text{with respect to the additive character $\bfe$, if $\star=\mathrm{hol}$},\\
\varphi_{F, \det\inv\otimes\rho(r)v}, & \text{with respect to the additive character $\overline{\bfe}$, if $\star=\mathrm{skew}$},
\end{cases}
\end{align*}
for any $r\in K_\infty$.
Here, recall that the identification $K_\infty = \U(n)$ depends on the choice of the additive character of $\R$.

As \cite[Theorem 7.5.5]{bs}, we have the following lemma.
\begin{lem}\label{adlif2}
Let $F \in J_{\rho_{\bfk}, 1}^{\star, \mathrm{cusp}}$ be a Hecke eigenform, where $\star\in\{\mathrm{hol}, \mathrm{skew}\}$.
Then the cuspidal automorphic representation $\pi_F \subset L^2_{\mathrm{cusp}}(\Sp^J_{2n}(\Q) \backslash \Sp^J_{2n}(\A))_{\psi}$ is irreducible.
Let us write $\pi_F=\bigotimes_v \pi_{F,v}$ for the decomposition of it into local components.
Then we have
\begin{itemize}
\item $\pi_{F, p}$ is unramified for all prime numbers $p$;
\item $\pi_{F, \infty}$ is isomorphic to
\begin{align*}
\begin{cases}
\pi_{\bfk-\frac{\mathbf{1}}{2}, \bfe} \otimes \pi_{\SW, \bfe},& \text{if $\star=\mathrm{hol}$},\\
\pi_{\bfk-\frac{\mathbf{1}}{2}, \overline{\bfe}} \otimes \pi_{\SW, \bfe},& \text{if $\star=\mathrm{skew}$},
\end{cases}
\end{align*}
where $\frac{\mathbf{1}}{2}$ denotes $(\frac{1}{2}, \ldots, \frac{1}{2})$.
\end{itemize}
\end{lem}
\begin{proof}
We shall prove it following the idea of Berndt-Schmidt \cite[Theorem 7.5.5]{bs}.
Since $\pi_F$ is cuspidal, it is a direct sum $\oplus_i \pi_i$ of finite numbers of irreducible cuspidal automorphic representations $\pi_i=\otimes_i \pi_{i,v}$.
As in the proof of \cite[Theorem 7.5.5]{bs}, we can specify the local components $\pi_{i, \infty}$ at the real place by using Lemma \ref{lw}.
In particular, all $\pi_{i,\infty}$ are isomorphic.
At any finite place $p$, by Theorem \ref{csu}, the local components $\pi_{i,p}$ are unramified.
As in the proof of \cite[Theorem 7.5.5]{bs}, their Satake parameters are completely determined by the Hecke eigenvalues of $F$.
In particular, $\pi_{i, p}$ are isomorphic for all $i$.
Thus $\pi_F$ is isotypic.

The $\Sp^J_{2n}(\hat{\Z})$-fixed part of $\pi_F$ is an isotypic $(\fraksp^J_{2n}, K_\infty)$-module generated by $\varphi_{F, v}$.
Hence it is irreducible, and so is $\pi_F$.
\end{proof}

Here we finish the proof of Lemma \ref{adlif}.
\begin{lem}\label{cut}
Let $F \in J_{\rho, 1}^{\mathrm{skew}, \mathrm{cusp}}$, $g \in \Sp^J_{2n}(\R)$, and put $H=F|_\rho^\mathrm{skew} g$.
Then we have
\begin{align*}
(8\pi i\frac{\partial}{\partial Z}-\frac{\partial^2}{\partial w^2})H (i,0)
=0
\end{align*}
\end{lem}
\begin{proof}
By the Iwasawa decomposition of $\Sp_{2n}(\R)$, we may assume that $g$ has a form $\mmatrix{A}{B}{0}{\tp{A}\inv}([\lambda,\mu],\kappa)$.
Since any skew-holomorphic Jacobi form $F(Z,w)$ can be written as
\begin{align*}
F(Z,w)
=\sum_{\nu\in(\Z/2\Z)^n} F_\nu(Z) \sum_{x \in \Z^n+\frac{\nu}{2}} \bfe(\tp{x}Zx+2\tp{x}w),
\end{align*}
where $F_\nu(Z)$ are anti-holomorphic functions, we have
\begin{align*}
H(Z,w)
&=\bfe(\tp{\lambda}Z\lambda+2\tp{\lambda}w+\tp{\lambda}\mu+\kappa) \frac{|\det \tp{A}\inv|}{\det \tp{A}\inv} \overline{\rho(\tp{A}\inv)}\inv F((AZ+B)\tp{A},A(w+Z\lambda+\mu))\\
&=\bfe(\tp{\lambda}Z\lambda+2\tp{\lambda}w+\tp{\lambda}\mu+\kappa) \sgn(\det A) \rho(\tp{A}) \\
&\quad \times\sum_{\nu\in(\Z/2\Z)^n} F_\nu((AZ+B)\tp{A}) \sum_{x \in \Z^n+\frac{\nu}{2}} \bfe(\tp{x}(AZ+B)\tp{A} x+2\tp{x}A(w+Z\lambda+\mu))\\
&=
\sum_{\nu\in(\Z/2\Z)^n} \sgn(\det A)\bfe(\tp{\lambda}\mu+\kappa) \rho(\tp{A}) F_\nu((AZ+B)\tp{A})\\
&\quad\times\sum_{x \in \Z^n+\frac{\nu}{2}}
\bfe(\tp{x}B\tp{A}x+2\tp{x}A\mu)
\bfe(\tp(\tp{A}x+\lambda)Z(\tp{A}x+\lambda)+2\tp(\tp{A}x+\lambda)w)\\
&=
\sum_{\nu\in(\Z/2\Z)^n} H_\nu(Z)
\sum_{x \in \tp{A}(\Z^n+\frac{\nu}{2})+\lambda}
c_x\bfe(\tp{x}Zx+2\tp{x}w),
\end{align*}
where $H_\nu(Z)=\sgn(\det A)\bfe(\tp{\lambda}\mu+\kappa) \rho(\tp{A}) F_\nu((AZ+B)\tp{A})$ are anti-holomorphic functions, and $c_x=\bfe(\tp(x-\lambda)A\inv B(x-\lambda)+2\tp(x-\lambda)\mu)$ are coefficients.
Then, relations
\begin{align*}
&\frac{\partial(\tp{x}Zx)}{\partial Z}=x\tp{x},&
&\frac{\partial(\tp{x}w)}{\partial w}=\tp{x},&
\end{align*}
imply that
\begin{equation*}
(8\pi i\frac{\partial}{\partial Z}-\frac{\partial^2}{\partial w^2})H (i,0)=0.
\end{equation*}
\end{proof}
\subsection{The adelic lift of $F' \in J_{(k,j),1}^{\mathrm{hol}, \mathrm{cusp}}$ or $J_{(k,j),1}^{\mathrm{skew}, \mathrm{cusp}}$}
As in \S \ref{adlifato}, let $j\geq0$ be an even integer, $k\geq3$ a positive integer.
As in the last subsection, let $\psi : \Q \backslash \A \to \C^1$ be the nontrivial character such that $\psi_\infty=\bfe$.
Let $F' \in J_{(k,j),1}^{\star, \mathrm{cusp}}$ be a Hecke eigenform, where $\star=\mathrm{hol}$ if $l\equiv k \pmod 2$ and $\star=\mathrm{skew}$ if $l\equiv k-1 \pmod 2$.
Put $F=\Psi(F') \in S_{k-\frac{1}{2},j}^+(\Gamma_0(4), \left(\frac{-1}{\cdot}\right)^l)$, which is the corresponding half-integral weight cusp form via the isomorphism $\Psi$ of Theorem \ref{5.1}.
Of course $F$ is also a Hecke eigenform.
Let $\pi_{F'}$ be the irreducible cuspidal automorphic representation of $\Sp^J_4(\A)$ with central character $\psi$ constructed in \S \ref{aljg}, and $\pi'=\bigotimes_v \pi'_v$ the irreducible cuspidal automorphic representation of $\Mp_4(\A)$ such that $\pi_{F'} \cong \pi' \otimes \pi_{\SW, \psi}$ (Corollary \ref{ccg}).
Then we have the following lemmas.
\begin{lem}\label{lfn}
We have
\begin{align*}
L(s, F)
=L(s-k-j+\tfrac{3}{2}, \pi', \overline{\psi}).
\end{align*}
\end{lem}
\begin{proof}
Let $p$ be a prime number.
It suffices to show that
\begin{align*}
L(s, F)_p
=L(s-k-j+\tfrac{3}{2}, \pi'_p, \overline{\psi}_p).
\end{align*}
By Lemma \ref{adlif2} and Theorem \ref{ccu}, the representation $\pi'_p$ is unramified with respect to $\overline{\psi}_p$.
Let $(\alpha_1, \alpha_2)$ be the Satake parameter of $\pi'_p$, i.e.,
\begin{align*}
\pi'_p &=\pi_{\overline{\psi}_p}(\alpha_1,\alpha_2),&
\pi_{F',p} &=\pi(\alpha_1, \alpha_2; \psi_p).
\end{align*}
Since a complete system of representatives $\{g_{s,t}\}_t$ of
\begin{align*}
\Sp_4(\Z) K_s(p^2) \Sp_4(\Z)
=\bigsqcup_t \Sp_4(\Z) g_{s,t}
\quad (s=1,2),
\end{align*}
is explicitly given in \cite[p.143]{az} and $\Phi_{F'}$ gives the spherical vector of $\pi_{F',p}$, the eigenvalues of $T^J_s(p)$ ($s=0, 1, 2$) associated with $F'$ can be expressed explicitly in terms of $\alpha_1$ and $\alpha_2$.
We shall write $\omega^J(p)$ and $\eta^J(p)$ for the eigenvalues of $T^J_0(p)$ and $T^J_1(p)$, respectively.
Then it can be shown by straightforward calculation that
\begin{align*}
\omega^J(p)
&=p^6(\alpha_1\alpha_2+\alpha_1\alpha_2\inv+\alpha_1\inv\alpha_2+\alpha_1\inv\alpha_2\inv+1-p^{-2}),\\
\eta^J(p)
&=p^\frac{11}{2}(\alpha_1+\alpha_2+\alpha_1\inv+\alpha_2\inv).
\end{align*}
Here, note that the calculations are algebraic, and hence independent of whether $F'$ is holomorphic or not.
Combining this with Theorem \ref{5.1}, we obtain
\begin{align*}
L(s,F)_p
&=\left\{
(1-\alpha_1 p^{-(s-k-j+\tfrac{3}{2})}) (1-\alpha_2 p^{-(s-k-j+\tfrac{3}{2})})
(1-\alpha_1\inv p^{-(s-k-j+\tfrac{3}{2})}) (1-\alpha_2\inv p^{-(s-k-j+\tfrac{3}{2})})
\right\}\inv\\
&=L(s-k-j+\tfrac{3}{2}, \pi'_p, \overline{\psi}_p).
\end{align*}
\end{proof}

\begin{lem}\label{ds}
The representation $\pi'_\infty$ is discrete series, and its $L$-parameter $\varphi'_\infty$ is
\begin{align*}
\varphi'_\infty=\calD_{k+j-\frac{3}{2}} \oplus \calD_{k-\frac{5}{2}},
\end{align*}
with respect to both $\psi_\infty$ and $\overline{\psi}_\infty$.
Moreover, the corresponding character of $S_{\varphi'_\infty} \cong (\Z/2\Z)^2$ is
\begin{align*}
(\Z/2\Z)^2 \ni (c,d) \mapsto
\begin{cases}
(-1)^c,& \text{$\star=\mathrm{hol}$},\\
(-1)^d,& \text{$\star=\mathrm{skew}$},
\end{cases}
\end{align*}
with respect to $\overline{\psi}_\infty$.
\end{lem}
\begin{proof}
Since $\det^k \Sym_j=\rho_{(k+j,k)}$, by Lemma \ref{adlif2} we have
\begin{align}\label{rr}
\pi'_\infty \cong
\begin{cases}
\pi_{(k+j-\frac{1}{2},k-\frac{1}{2}), \bfe},& \text{if $\star=\mathrm{hol}$},\\
\pi_{(k+j-\frac{1}{2},k-\frac{1}{2}), \overline{\bfe}},& \text{if $\star=\mathrm{skew}$}.
\end{cases}
\end{align}
Then the assertion follows from \cite[\S C.2.1]{gi20}.
\end{proof}

\begin{lem}\label{apj}
If $l=1$, then there is a unique irreducible everywhere unramified cuspidal symplectic automorphic representation $\tau_F$ of $\GL_4(\A)$ such that
\begin{align*}
L(s,F)
&=L(s-j-k+\tfrac{3}{2}, \tau_F),\\
\tau_{F,\infty}
&=\calD_{k+j-\frac{3}{2}} \oplus \calD_{k-\frac{5}{2}}.
\end{align*}

On the other hand, if $l=0$, then there exists one and only one of such a $\tau_F$ or a pair $(\sigma_F, \sigma_F')$ of irreducible cuspidal symplectic automorphic representations $\sigma_F$ and $\sigma_F'$ of $\GL_2(\A)$ such that
\begin{gather*}
L(s,F)
=L(s-j-k+\tfrac{3}{2}, \sigma_F)L(s-j-k+\tfrac{3}{2}, \sigma_F'),\\
\sigma_{F,\infty}
=\calD_{k+j-\frac{3}{2}}, \qquad
\sigma_{F,\infty}'
=\calD_{k-\frac{5}{2}}.
\end{gather*}
\end{lem}
\begin{proof}
Let $\phi_F$ be the $A$-parameter of $\pi'$ relative to $\overline{\psi}$.
The $A$-parameter $\phi_F$ for $\Mp_4$ is one of the following forms:
\begin{enumerate}[(1)]
\item $\phi_F=\chi \boxtimes S_4$, where $\chi$ is a quadratic character of $\Q^\times \backslash \A^\times$;
\item $\phi_F=\chi \boxtimes S_2 \oplus \chi' \boxtimes S_2$, where $\chi$ and $\chi'$ are distinct quadratic characters of $\Q^\times \backslash \A^\times$;
\item $\phi_F=\sigma \boxtimes S_2$, where $\sigma$ is an irreducible cuspidal orthogonal automorphic representation of $\GL_2(\A)$;
\item $\phi_F=\chi \boxtimes S_2 \oplus \sigma \boxtimes S_1$, where $\chi$ is a quadratic character of $\Q^\times \backslash \A^\times$, and $\sigma$ is an irreducible cuspidal symplectic automorphic representation of $\GL_2(\A)$;
\item $\phi_F=\sigma \boxtimes S_1 \oplus \sigma' \boxtimes S_1$, where $\sigma$ and $\sigma'$ are distinct irreducible cuspidal symplectic automorphic representations of $\GL_2(\A)$;
\item $\phi_F=\tau \boxtimes S_1$, where $\tau$ is an irreducible cuspidal symplectic automorphic representation of $\GL_4(\A)$.
\end{enumerate}

If the case is (1) or (2), then by \cite{gi20}, the local $A$-packet $\Pi_{\phi_{F, \infty}, \overline{\psi}_\infty}(\Mp_4(\R))$ does not contain any discrete series representation.
This contradicts Lemma \ref{ds}, and thus the cases (1) and (2) cannot occur.

If the case is (3), the local component $\sigma_\infty$ of $\sigma$ at the real place must be irreducible by the same reason why the case (2) cannot occur.
Then by the table in \cite[\S C.2.2]{gi20}, the discrete series representation $\pi'_\infty \in \Pi_{\phi_{F, \infty}, \overline{\psi}_\infty}(\Mp_4(\R))$ corresponds to the nontrivial character of $S_{\phi_{F, \infty}} \cong \Z/2\Z$.
On the other hand, local components $\pi'_p$ at the finite places are unramified.
These contradict the fact that Gan-Ichino's character \cite[\S 2.1]{gi20} $\wtil{\epsilon}_{\phi_F}$ is trivial, so the case (3) is impossible.

Assume that $\phi_F$ is of the form (4).
If the local component $\sigma_\infty$ of $\sigma$ at the real place is reducible, then the nonzero elements of the $A$-packet $\Pi_{\phi_{F, \infty}, \overline{\psi}_\infty}(\Mp_4(\R))$ associated to $\phi_{F, \infty}$ (relative to $\overline{\psi}_\infty$) are $\pi^{+,+}$ and $\pi^{+,-}$ in the notation of \cite[\S 8.1]{gi20}.
However, by \cite[\S 8.1, Lemma C8]{gi20}, neither one is a discrete series representation.
Thus $\sigma_\infty$ is irreducible.
Since $\pi'$ is unramified at every finite place, so is $\chi$.
Therefore $\chi$ is the trivial character of $\A^\times$, and Gan-Ichino's character $\wtil{\epsilon}_{\phi_F}$ is trivial on $\{0\} \oplus (\Z/2\Z) \subset S_{\phi_F} \cong (\Z/2\Z)^2$.
Combining this with the table in \cite[\S C.2.2]{gi20}, we can see that $k=3$, $\sigma_\infty=\calD_{j+\frac{3}{2}}$, and the $A$-packet $\Pi_{\phi_{F, \infty}, \overline{\psi}_\infty}(\Mp_4(\R))$ contains only one discrete series representation that is an element of the $L$-packet $\Pi_{\varphi'_\infty, \overline{\psi}_\infty}(\Mp_4(\R))$ and corresponds to the trivial character of $S_{\varphi'_\infty}$.
This contradicts Lemma \ref{ds}.

Assume that $\phi_F$ is of the form (5).
Then by the disjointness of $L$-packets, we have $\phi_{F, \infty}=\varphi'_\infty$.
We may assume that
\begin{align*}
\sigma_\infty &=\calD_{k+j-\frac{3}{2}},&
\sigma'_\infty &=\calD_{k-\frac{5}{2}}.
\end{align*}
Since $\pi'$ is unramified at every finite place, so are $\sigma$ and $\sigma'$.
Thus we have
\begin{align*}
\epsilon(\tfrac{1}{2}, \sigma)
&=\epsilon(\tfrac{1}{2}, \calD_{k+j-\frac{3}{2}}, \psi_\infty)
=i^{2k+2j-3+1}
=(-1)^{k-1},\\
\epsilon(\tfrac{1}{2}, \sigma')
&=\epsilon(\tfrac{1}{2}, \calD_{k-\frac{5}{2}}, \psi_\infty)
=i^{2k-5+1}
=(-1)^k,
\end{align*}
since $j$ is even.
Therefore, by Gan-Ichino's multiplicity formula, the corresponding character of $\pi'_\infty$ must be
\begin{align*}
S_{\varphi'_\infty} \cong (\Z/2\Z)^2 \ni (c,d) \mapsto
\begin{cases}
(-1)^d, & \text{if $k$ is odd},\\
(-1)^c, & \text{if $k$ is even}.
\end{cases}
\end{align*}
If $l=1$, this contradicts Lemma \ref{ds}.

Consequently, when $l=1$, the $A$-parameter $\phi_F$ of $\pi'$ relative to $\overline{\psi}$ is of the form $\tau_F \boxtimes S_1$ as in the case (6).
On the other hand, it is of the form $\sigma_F \boxtimes S_1 \oplus \sigma_F' \boxtimes S_1$ or $\tau_F \boxtimes S_1$ as in the cases (5) and (6).
Then by Lemmas \ref{lfn} and \ref{ds}, $\tau_F$ satisfies the conditions.

As Lemma \ref{api}, the uniqueness follows from that of the $A$-parameter.
\end{proof}

Since $\tau_F$ or $(\sigma_F, \sigma_F')$ is uniquely determined by $F$, we obtain an assignment $F \mapsto \tau_F$ or $(\sigma_F, \sigma_F')$.
As in \S \ref{adlifato}, we have the following lemma.
\begin{lem}\label{jap}
Let $\tau$ be an irreducible cuspidal symplectic automorphic representation of $\GL_4(\A)$ that is unramified everywhere and satisfies $\tau_\infty=\calD_{k+j-\frac{3}{2}} \oplus \calD_{k-\frac{5}{2}}$.
Then there exists a Hecke eigenform $F \in S_{k-\frac{1}{2},j}^+(\Gamma_0(4), \left(\frac{-1}{\cdot}\right)^l)$ such that $\tau_F=\tau$, and it is unique up to a scalar multiple.

Moreover, when $l=0$, let $(\sigma, \sigma')$ be a pair of irreducible cuspidal symplectic automorphic representations of $\GL_2(\A)$ that are unramified everywhere and satisfy $\sigma_\infty=\calD_{k+j-\frac{3}{2}}$ and $\sigma'_\infty=\calD_{k-\frac{5}{2}}$.
Then there exists a Hecke eigenform $F \in S_{k-\frac{1}{2},j}^+(\Gamma_0(4))$ such that $(\sigma_F, \sigma_F')=(\sigma, \sigma')$, and it is unique up to a scalar multiple.
\end{lem}
\begin{proof}
We have Gan-Ichino's multiplicity formula (Theorem \ref{gimf}) for $L^2_{\mathrm{disc}}(\Mp_4)$ and canonical bijective correspondences (Theorems \ref{ccp}, \ref{ccr}, and \ref{cca}) between representations of the metaplectic groups and the Jacobi groups.
Also, in Lemma \ref{apj}, we have already found the $A$-parameter corresponding to any Hecke eigenform in $J_{(k,j),1}^{\star, \mathrm{cusp}}$.
Combining this with Theorem \ref{5.1}, we can prove the lemma in a similar way to \cite[Proposition 9.1.4, (i) and (iii)]{cl}.
\end{proof}
\subsection{Proofs of Ibukiyama's conjectures}\label{Proof}
In this subsection, we finally give proofs of Ibukiyama's conjectures (Theorem \ref{main}, \ref{lifting}, and \ref{compl}).\\

First, we shall prove the Shimura type isomorphism on the Neben type (Theorem \ref{main}).
Since the spaces $S_{k-\frac{1}{2},j}^+(\Gamma_0(4), \left(\frac{-1}{\cdot}\right))$ and $S_{j+3, 2k-6}(\Sp_4(\Z))$ have bases consisting of Hecke eigenforms, the assertion follows from Lemmas \ref{api}, \ref{iap}, \ref{apj}, and \ref{jap}.

\begin{rem}
As pointed out by Ibukiyama \cite{ibuconj}, the Shimura type conjecture on the Neben type is false when $j$ is odd.
This can be confirmed by the multiplicity formulae.
By the proof of Lemma \ref{api}, Arthur's multiplicity formula shows the existence of Hecke eigenform in $S_{j+3, 2k-6}(\Sp_4(\Z))$ even if $j$ is odd.
However, the proof of Lemma \ref{apj} implies that Gan-Ichino's multiplicity formula shows that there is no automorphic representation corresponding to a Hecke eigenform in $S_{k-\frac{1}{2},j}^+(\Gamma_0(4), \left(\frac{-1}{\cdot}\right))$, i.e., the plus space is zero.\\
\end{rem}

Next, let us give a proof of Theorem \ref{lifting}.
As we have remarked after stating Theorem \ref{lifting}, we may assume that $k$ is greater than 7.
For any Hecke eigenforms $f \in S_{2k-4}(\SL_2(\Z))$ and $g \in S_{2k+2j-2}(\SL_2(\Z))$, let $\tau_f$ and $\tau_g$ be the corresponding automorphic representation of $\GL_2(\A)$, respectively.
Then it is well known that the following properties hold:
\begin{itemize}
\item $\tau_f$ and $\tau_g$ are cuspidal;
\item $\tau_{f, \infty}=\calD_{k-\frac{5}{2}}$ and $\tau_g=\calD_{k+j-\frac{3}{2}}$;
\item $\tau_{f,p}$ and $\tau_{g,p}$ are unramified for all prime number $p$;
\item $L(s-j-1,f)=L(s-j-k+\frac{3}{2}, \tau_f)$ and $L(s,g)=L(s-j-k+\frac{3}{2}, \tau_g)$.
\end{itemize}
Therefore, by Lemma \ref{jap}, there exists a Hecke eigenform $F \in S_{k-\frac{1}{2},j}^+(\Gamma_0(4))$ such that $L(s,F)=L(s-j-1,f)L(s,g)$.
This gives us an injective linear map
\begin{equation*}
\calL : S_{2k-4}(\SL_2(\Z)) \otimes S_{2k+2j-2}(\SL_2(\Z)) \lra S_{k-\frac{1}{2},j}^+(\Gamma_0(4))
\end{equation*}
such that if $f \in S_{2k-4}(\SL_2(\Z))$ and $g \in S_{2k+2j-2}(\SL_2(\Z))$ are Hecke eigenforms, then so is $\calL(f\otimes g) \in S_{k-\frac{1}{2},j}^+(\Gamma_0(4))$, and they satisfy
\begin{equation*}
L(s, \calL(f\otimes g)) = L(s-j-1, f) L(s,g).
\end{equation*}

Now we come to the Shimura type isomorphism on Haupt type (Theorem \ref{compl}).
If $(\sigma, \sigma')$ is a pair of irreducible cuspidal symplectic automorphic representations of $\GL_2(\A)$ that are unramified everywhere and satisfy $\sigma_\infty=\calD_{k+j-\frac{3}{2}}$ and $\sigma'_\infty=\calD_{k-\frac{5}{2}}$, then it is known that there exist Hecke eigenforms $f \in S_{2k-4}(\SL_2(\Z))$ and $g \in S_{2k+2j-2}(\SL_2(\Z))$ such that the corresponding automorphic representations of $\GL_2(\A)$ are $\sigma'$ and $\sigma$, respectively.
Thus by the proof of Theorem \ref{lifting}, it follows from Lemmas \ref{apj} and \ref{jap} that the image of $\calL$ is spanned by the Hecke eigenforms $F \in S_{k-\frac{1}{2},j}^+(\Gamma_0(4))$ corresponding to the pairs $(\sigma_F, \sigma_F')$ of irreducible cuspidal symplectic automorphic representations of $\GL_2(\A)$, and its orthogonal complement $S_{k-\frac{1}{2},j}^{+,0}(\Gamma_0(4))$ is spanned by those corresponding to irreducible cuspidal symplectic automorphic representations $\tau_F$ of $\GL_4(\A)$.
Combining this with Lemmas \ref{api} and \ref{iap}, we obtain Theorem \ref{compl}.

\appendix

\section{The adelic lift of $F \in S_{k-\frac{1}{2}, j}^+(\Gamma_0(4), (\frac{-1}{\cdot})^l)$}\label{Appendix}
In this appendix, we consider the adelic lifts of Siegel cusp forms of half-integral weight.
Note that any argument here is not needed to prove Theorem \ref{main}.
Let $j\geq0$ and $k\geq3$ be integers, $l \in \Z/2\Z$,  and $\psi$ the nontrivial additive character of $\Q \backslash \A$ such that $\psi_\infty=\bfe$.
If $j$ is odd, then the plus space $S_{k-\frac{1}{2},j}^+(\Gamma_0(4), \left(\frac{-1}{\cdot}\right)^l)$ is zero, and the arguments in this section are trivially true.
Hence we assume that $j$ is even.

For any prime number $p$, the restriction of the Weil representation $\omega_{\psi_p}$ of $\Mp_4(\Q_p)$ to the metaplectic covering $\wtil{\Gamma}_0(4)_p$ over
\begin{align*}
\Gamma_0(4)_p
=\Set{
\mmatrix{A}{B}{C}{D} \in \Sp_4(\Z_p)
|
C \in 4\Z_p
}
\end{align*}
defines a genuine character $\varepsilon_p$ of $\wtil{\Gamma}_0(4)_p$ by
\begin{align*}
\omega_{\psi_p}(\gamma) 1_{\Z_p^2}
&=\varepsilon_p(\gamma)\inv 1_{\Z_p^2},&
\gamma&\in\wtil{\Gamma}_0(4)_p.
\end{align*}
Note that if $p\neq2$, $\varepsilon_p$ is quadratic and defines the splitting \eqref{spl} over $\Gamma_0(4)_p=\Sp_4(\Z_p)$.
Then, these characters define a genuine character $\varepsilon_\fin=\prod_p \varepsilon_p$ of a subgroup
\begin{align*}
\wtil{\Gamma}_0(4)_\fin = \wtil{\Gamma}_0(4)_2 \times \prod_{p\neq2} \Gamma_0(4)_p
\end{align*}
of $\Mp_4(\A)$.
At the real place, the Weil representation defines the factor of automorphy $j_\infty(g,Z)$ by
\begin{equation*}
\omega_{\psi_\infty}(g)\varphi_Z
=j_\infty(g,Z)\inv \varphi_{gZ},
\end{equation*}
for $g \in \Mp_4(\R)$ and $Z \in \frakH_2$, where
\begin{equation*}
\varphi_Z(x)
=\bfe(\tp{x}Zx) \in \calS(\R^2).
\end{equation*}
Note that the weight of the character $j_\infty(-, i) : \wtil{K_\infty} \to \C^\times$ relative to the additive character $\psi_\infty=\bfe$ is $(\frac{1}{2}, \ldots, \frac{1}{2})$.
Then it is known that for any $\gamma \in \Gamma_0(4)$, we have
\begin{align*}
\frac{\theta(\gamma Z)}{\theta(Z)}
=j_\infty((\gamma,1),Z) \varepsilon_\fin((\gamma,1)).
\end{align*}

The strong approximation theorem for $\Sp_4$ induces that for $\Mp_4$, hence we have
\begin{equation*}
\Mp_4(\A)
=\Sp_4(\Q) \Mp_4(\R) \wtil{\Gamma_0}(4)_\fin.
\end{equation*}
Let $F\in S_{k-\frac{1}{2},j}^+(\Gamma_0(4), \left(\frac{-1}{\cdot}\right)^l)$ be a Siegel modular form.
Now the adelic lift $\Phi_F$ of $F$ is defined by
\begin{align*}
\Phi_F(g)
=\left( \frac{-1}{\kappa} \right)^l \left\{ j_\infty(g_\infty, i) \varepsilon_\fin(\kappa) \right\}^{-(2k-1)} \Sym_j(J(g_\infty,i))\inv F(g_\infty i),
\end{align*}
where $g=\gamma g_\infty \kappa \in \Mp_4(\A)=\Sp_4(\Q) \Mp_4(\R) \wtil{\Gamma_0}(4)_\fin$.
Here, note that $\left(\frac{-1}{\kappa} \right)=\left(\frac{-1}{\kappa_2} \right)$, ($\kappa=(\kappa_p)_p$).
The function $\Phi_F$ is well-defined.

\begin{prop}\label{alh}
The function $\Phi_F : \Mp_4(\A) \to V_j$ satisfies the following:
\begin{enumerate}[(1)]
\item $\Phi_F$ is left $\Sp_4(\Q)$-invariant;
\item $\Phi_F$ is right $\Sp_4(\Z_p)$-invariant for $p\neq2$;
\item $\Phi_F(gx) = j_\infty(x,i)^{-2k+1} \Sym_j(J(x,i))\inv \Phi_F(g)$, for any $g\in\Mp_4(\R)$, $x\in \wtil{K_\infty}$;
\item $\frakp_\C^- \cdot \Phi_F = 0$;
\item $\Phi_F$ is cuspidal.
\end{enumerate}
\end{prop}
\begin{proof}
The proof is similar to that of \cite[Theorem 1]{as}.
\end{proof}

Since the complex conjugate $\overline{\Sym_j}$ is isomorphic to the contragredient representation of $\Sym_j$, by a similar argument to \cite[\S 4.5 and \S 6.3.4]{cl}, we can construct an automorphic cuspidal representation
\begin{align*}
\pi_F=\bigotimes_v \pi_{F,v} \subset L^2_\mathrm{disc}(\Mp_4).
\end{align*}
and check that it is a direct sum $\oplus_i \pi_i$ of a finite number of irreducible automorphic cuspidal representations $\pi_i=\otimes_v \pi_{i,v}$ such that
\begin{itemize}
\item $\pi_{i, p}$ are unramified and  isomorphic to each other for any odd prime $p$;
\item $\pi_{i,\infty} =\pi_{(k+j-\frac{1}{2}, k-\frac{1}{2}), \bfe} =\pi_{(-k+\frac{1}{2}, -j-k+\frac{1}{2}), \overline{\bfe}}$  for any $i$.
\end{itemize}
In particular, $\pi_i$ are nearly equivalent.
Moreover, we have the following lemma.
\begin{lem}\label{lp}
Let $p$ be an odd prime.
\begin{enumerate}[(1)]
\item Assume that $l\equiv k \pmod 2$, and let $(\alpha_1, \alpha_2)$ be the Satake parameter of $\pi_{i, p}$ with respect to $\overline{\psi}_p$.
Then we have
\begin{align*}
\eta(p)
&=\left(\frac{-1}{p}\right)^k p^2(\alpha_1+\alpha_2+\alpha_1\inv+\alpha_2\inv),\\
\omega(p)
&=p^3(\alpha_1\alpha_2+\alpha_1\alpha_2\inv+\alpha_1\inv\alpha_2+\alpha_1\inv\alpha_2\inv+1-p^{-2}),
\end{align*}
and hence
\begin{align*}
L(s, F)_p
=L(s-k-j+\tfrac{3}{2}, \pi_{i,p}, \overline{\psi}_p).
\end{align*}
\item Assume that $l\equiv k-1 \pmod 2$, and let $(\alpha_1, \alpha_2)$ be the Satake parameter of $\pi_{i, p}$ with respect to $\psi_p$.
Then we have
\begin{align*}
\eta(p)
&=\left(\frac{-1}{p}\right)^{k+1} p^2(\alpha_1+\alpha_2+\alpha_1\inv+\alpha_2\inv),\\
\omega(p)
&=p^3(\alpha_1\alpha_2+\alpha_1\alpha_2\inv+\alpha_1\inv\alpha_2+\alpha_1\inv\alpha_2\inv+1-p^{-2}),
\end{align*}
and hence
\begin{align*}
L(s,F)_p
=L(s-k-j+\tfrac{3}{2}, \pi_{i,p}, \psi_p).
\end{align*}
\end{enumerate}
\end{lem}
\begin{proof}
Complete systems $\{\wtil{g}_{s,t}\}_t$ of representatives of the right coset decompositions
\begin{align*}
\wtil{\Gamma}_0(4) (K_s(p^2), p^{1-\frac{s}{2}}) \wtil{\Gamma}_0(4)
=\bigsqcup_t \wtil{\Gamma}_0(4) \wtil{g}_{s,t}
\end{align*}
are explicitly given by Zhuravlev \cite{zhu}.
Thus the assertions are proved by straightforward calculation.
\end{proof}

Let us recall from Theorem \ref{5.1} that we have an isomorphism $\Psi$.
Put
\begin{align*}
F'=\Psi\inv(F) \in
\begin{cases}
J_{(k,j),1}^\mathrm{hol, cusp},&\text{ when $l\equiv k \pmod 2$},\\
J_{(k,j),1}^\mathrm{skew, cusp},&\text{ when $l\equiv k-1 \pmod 2$},
\end{cases}
\end{align*}
and write
\begin{align*}
\pi_{F'}
=\pi' \otimes \pi_{\SW, \psi}
=\bigotimes_v \left( \pi'_p \otimes \pi_{\SW, \psi_p} \right).
\end{align*}
Then $\pi_F$ and $\pi'$ are related as follows.
\begin{thm}\label{appthm}
The automorphic representation $\pi_F$ is irreducible, and if we write
\begin{align*}
\pi_F=\bigotimes_v \pi_{F, v},
\end{align*}
then we have the following.
\begin{enumerate}[(1)]
\item Assume that $l\equiv k \pmod 2$.
Then we have $\pi_{F, v}\cong \pi'_v$ for every place $v$, i.e., $\pi_F \cong \pi'$.
\item Assume that $l\equiv k-1 \pmod 2$.
Then the $L$-parameter and the character of its component group associated to $\pi_{F,v}$ relative to $\psi_v$  coincide with those associated to $\pi'_v$ relative to $\overline{\psi}_v$.
\end{enumerate}
\end{thm}
\begin{proof}
(1)
At the real place, by \eqref{rr} we have
\begin{align*}
\pi_{i,\infty} \cong \pi_{(k+j-\frac{1}{2}, k-\frac{1}{2}), \bfe} \cong \pi'_\infty.
\end{align*}
At every finite place $p$ except $2$, by Lemmas \ref{lfn} and \ref{lp} we have
\begin{align*}
L(s-k-j+\tfrac{3}{2}, \pi_{i,p}, \overline{\psi}_p)
=L(s, F)_p
=L(s-k-j+\tfrac{3}{2}, \pi'_p, \overline{\psi}_p).
\end{align*}
Since both of $\pi_{i,p}$ and $\pi'_p$ are unramified representation of $\Mp_4(\Q_p)$, they are isomorphic.
In particular, $\pi_i$ and $\pi'$ are nearly equivalent, and hence the $A$-parameter of $\pi_i$ relative to $\overline{\psi}$ is that of $\pi'$ relative to $\overline{\psi}$.
By the proof of Lemma \ref{apj}, the $A$-parameter of $\pi'$ relative to $\overline{\psi}$ is tempered.
(Note that the case (5) may occur.)
Since the local $L$-packet of an unramified $L$-parameter for $\Mp_4$ is a singleton, this implies that
\begin{align*}
\pi_{i,2} \cong \pi'_2.
\end{align*}
Hence $\pi_F$ is isomorphic to a direct sum of $\pi'$.
Since Gan-Ichino's multiplicity formula tells us that $\pi'$ appears in $L^2_\mathrm{disc}(\Mp_4)$ with multiplicity one, the assertion follows.\\

(2)
At the real place, since
\begin{align*}
\pi_{i,\infty} &\cong \pi_{(k+j-\frac{1}{2}, k-\frac{1}{2}), \bfe},&
\pi'_\infty &\cong \pi_{(k+j-\frac{1}{2},k-\frac{1}{2}), \overline{\bfe}},
\end{align*}
the $L$-parameter and the character of its component group associated to $\pi_{i,\infty}$ relative to $\psi_\infty$ coincide with those associated to $\pi'_\infty$ relative to $\overline{\psi}_\infty$.
At every finite place $p$ except $2$, by Lemmas \ref{lfn} and \ref{lp} we have
\begin{align*}
L(s-k-j+\tfrac{3}{2}, \pi_{i,p}, \psi_p)
=L(s,F)_p
=L(s-k-j+\tfrac{3}{2}, \pi'_p, \overline{\psi}_p).
\end{align*}
Since the local $L$-packet of an unramified $L$-parameter for $\Mp_4$ is a singleton, there is a similar relation between $\pi_{i,p}$ and $\pi'_p$.
Hence the $A$-parameter of $\pi_i$ relative to $\psi$ is that of $\pi'$ relative to $\overline{\psi}$, which is tempered.
Then the assertion follows from the same argument as the proof of (1).
\end{proof}



%
\end{document}